\documentclass[11pt]{article} % 或 book、report 等

\usepackage{changes}
\usepackage{mathtools}
\usepackage{amssymb,latexsym}
\usepackage{fancyhdr,parskip}
\usepackage{graphicx,amsthm}
\usepackage{latexsym}
\usepackage{tasks}
\usepackage{overpic}
\usepackage{comment}
\usepackage{geometry}
\usepackage{scrextend}
\usepackage{listings,xcolor}
\usepackage{parskip}
\usepackage{mathrsfs} % use \mathscr 
\usepackage{cite} %使用参考文献
\usepackage{url} %当插入网址的时候就需要这个
\usepackage{framed} %边框
\usepackage{subfiles} %使用多个文件分开的时候用
\usepackage[utf8]{inputenc}
\usepackage{tcolorbox} 
\usepackage{hyperref}
\hypersetup{
}

\newcommand{\Line}[2]{\underline{\hbox to #1 cm{#2}}}
\newcommand{\pa}{\partial} % partial derivative
\newcommand{\xwl}{\overline} 

\newcommand{\Rr}{\mathbb{R}} % real field
\newcommand{\Zz}{\mathbb{Z}} % integer domain 
\newcommand{\xth}{\theta}

\renewcommand{\Re}{\mathrm{Re}} % renew real part
\renewcommand{\Im}{\mathrm{Im}} % renew image part
\numberwithin{equation}{section} %保证每一节的公式

\renewcommand\d{\partial}
\renewcommand\a{\alpha}

\newcommand\R{\mathbb R}

\def\la{\lambda}

\def\epsilon{\varepsilon}
\def\e{\varepsilon}

\newcommand\br{\begin{rem}}
\newcommand\er{\end{rem}}
\newcommand\bp{\begin{pmatrix}}
\newcommand\ep{\end{pmatrix}}
\newcommand\be{\begin{equation}}
\newcommand\ee{\end{equation}}
\newcommand\ba{\begin{equation}\begin{aligned}}
\newcommand\ea{\end{aligned}\end{equation}}

\newcommand\bite{\begin{itemize}}
\newcommand\eite{\end{itemize}}

\usepackage{layout}

\newcommand{\D}[1]{\,\mathrm{d}#1}

\newcommand{\De}{{\Delta}}

\newtheorem{definition}{Definition}[section]
\newtheorem{theorem}[definition]{Theorem}
\newtheorem{proposition}[definition]{Proposition}

\newtheorem{corollary}[definition]{Corollary}
\newtheorem{remark}[definition]{Remark}

\numberwithin{equation}{section}

\begin{document}

\setlength{\parindent}{2em}

\title{Error estimates in the non-relativistic limit  for the two-dimensional cubic Klein-Gordon equation}
\author{Yong Lu \footnote{School of Mathematics, Nanjing University, 22 Hankou Road, Gulou District, 210093 Nanjing, China. Email: luyong@nju.edu.cn. The research of Y. Lu was supported by Natural Science Foundation of Jiangsu Province (Grant No. BK20240058)} \and  Fangzheng Huang \footnote{school of Mathmatics, Nanjing University, 22 Hankou Road, Gulou Sistrict, 210093 Nanjing, China. Email:211250057@smail.nju.edu.cn.}}

\date{}
\maketitle

\begin{abstract}
   In this paper, we study the non-relativistic limit of the two-dimensional cubic nonlinear Klein-Gordon equation with a small parameter $0<\e \ll 1$ which is inversely proportional to the speed of light. We show the cubic nonlinear Klein-Gordon equation converges to the cubic nonlinear Schr\"{o}dinger equation with a convergence rate of order $O(\e^2)$. In particular, for the defocusing case with high regularity initial data, we show error estimates of the form $C(1+t)^N \e^2$ at time $t$ up to a long time of order $\e^{-\frac{2}{N+1}}$, while for initial data with limited regularity, we also show error estimates of the form $C(1+t)^M\e$ at time $t$ up to a long time of order $\e^{-\frac{1}{M+1}}$. Here $N$ and $M$ are constants depending on initial data. The idea of proof is to reformulate nonrelativistic limit problems to stability problems in geometric optics, then employ the techniques in geometric optics to construct approximate solutions up to an arbitrary order, and finally, together with the decay estimates of the cubic Schr\"{o}dinger equation, derive the error estimates.

\end{abstract}


\tableofcontents

\section{Introduction}

\subsection{Setting}
The Klein-Gordon equation is a relativistic version of the Schr\"{o}dinger equation to describe the motion of spinless particles. We consider the two-dimensional nonlinear Klein-Gordon equation in the dimensionless form
\begin{align} \label{K_L} 
    \e^{2} \pa_{t t}u - \De_{x}u + \e^{-2}u + f(u) = 0, \quad  t \geq 0,  \ x \in \Rr^{2}.
\end{align}
Here $u := u(t , x)$ and the nonlinearity are  real-valued functions. In this paper, we consider the classical cubic nonlinearity $f(u)= \la u^3$ with $\la \neq 0$. The small parameter $\e$ is inversely proportional to the speed of light. Our goal is to study the convergence of the nonlinear Klein-Gordon equation in the non-relativistic limit $ \e \to 0$ with initial data
\begin{align} \label{inital_KL}
    u(0,x)=\phi(x),\quad 
    ( \pa_{t} u )(0, x)= \e^{-2} \psi(x), \quad x\in \Rr^{2},
\end{align}
where $\phi$ and $\psi$ are real-valued functions.

For some fixed $\e$, the well-posedness of the nonlinear Klein-Gordon equation has been thoroughly studied, as seen in works like \cite{GV85} and \cite{GV89}. The non-relativistic limit of the Klein-Gordon equation \eqref{K_L}-\eqref{inital_KL} has been extensively studied in both theoretical and numerical fields, with notable contributions from references such as \cite{BCZ14, BD12, BZ16, BZ17, BZ19, BZ19_2, LZ17, M01,MNO02,N90, SZ20, MS72,MN02, WL23}.

Numerical simulations play an important role in studying the Klein-Gordon equation in the non-relativistic limit. %Recently, based on extensive numerical results. In \cite{BZ17}, Bao and Zhao give similar results for the Klein-Gordon-Schr\"{o}dinger system, while   in \cite{BZ16}, they give similar results for the Klein-Gordon-Zakharov system. 
Notably, numerical simulations in \cite{SZ20} and \cite{BZ16} show that in the non-relativistic regime, the Klein-Gordon equation converges to the nonlinear Schr\"{o}dinger equation with an error estimate of the form $(1 + t)\e^2$ for smooth initial data, and of the form $(1 + t)\e$ for insufficiently smooth initial data.

For the case of quadratic nonlinearities, a uniform error estimate of order $O(\e)$ over an order $O(\frac{1}{\e})$ long time  is obtained in \cite{LZ16}. For cubic nonlinearity and three-dimensional case, it is shown in \cite{BLZ24} that the error estimates are of the form $(1+t)\e^2$ for data with high regularity and of the form $(1+t)\e$ for data with limited regularity. The result is actually consistent with numerical result in \cite{BZ17} and \cite{SZ20}. This paper is devoted to the study of two-dimensional case. We will show that in the two-dimensional case, the error estimate exhibits a higher growth rate of the form $(1 + t)^{N_{\phi,\psi}}$, where $N_{\phi,\psi}$ is a fixed constant greater than or equal to $1$, depending on the initial data. Actually, in \cite{WL23}, Wu and Lei obtained error estimates of the form $ (1+t) \e^2$ in two-dimensional case in $L^2$ norm, whereas we get error estimates in $H^s$ norm with $s$  large as long as the initial data are sufficiently regular.

Throughout this paper, $C$ denotes a generic constant independent of the initial datum $(\phi, \psi)$;  $C_{\phi,\psi}$ denotes a constant depending on the initial datum $(\phi, \psi)$; $C_{\phi,\psi}(t)$ denotes a positive, continuous, and increasing function of time $t$. Specifically, $C$, $C_{\phi,\psi}$ and $C_{\phi,\psi}(t)$ are $\e$-independent, and their values may vary from line to line. 

\subsection{Main Results}
We will show that in the non-relativistic limit, the cubic nonlinear Klein-Gordon equation can be approximated by the cubic nonlinear Schr\"{o}dinger equation
\begin{align} \label{ls_g_0}
    2i \pa_{t} g_{0} - \De_{x} g_{0}  +3 \la \left| g_{0} \right| ^ {2} g_{0} =0, \quad g_{0} (0, \cdot) = \frac{\phi - i \psi}{2}.
\end{align}
We assume that the regularity on the initial data is given as
\begin{align} \label{inital_datum}
     \phi, \, \psi \in H^s(\R^2), \quad s \geq 1.
\end{align}
To get decay estimates for the solutions of the cubic Schr\"{o}dinger equation, we assume the following weighted regularity on the initial data
\begin{align} \label{weighted}
      \phi, \psi \in W^{1, \frac{4}{3}}(\Rr^2), \quad |x|\phi,\, |x|\psi \in H^1(\mathbb{R}^2) \cap L^{\frac{4}{3}}(\mathbb{R}^2).
\end{align}
We recall some results concerning the well-posedness of \eqref{ls_g_0}. The local well-posedness is classical, see for example \cite{B99}.
\begin{proposition}
\label{local_wp} %%这里需要的是 [] 
    Suppose $(\phi, \psi) $ satisfies \eqref{inital_datum} with $s > 1$. Then there exists a unique local solution $g_0(t, x) \in C([0, T^*), H^s(\Rr^2))$ to \eqref{ls_g_0} with $T^*$ depending on the initial datum $(\phi,\psi)$.
\end{proposition}
%blow up就是爆破的意思
For the focusing case ($\lambda < 0$), the solution to the cubic Schrödinger equation \eqref{ls_g_0} may blow up in finite time. However, for the defocusing case ($\lambda > 0$), the cubic Schrödinger equation is globally well-posed, and its solution admits long-time decay estimates (for global well-posedness, see Theorem 2 and Remark 2 in \cite{B99}; for decay estimates, see Theorem 6.1 in \cite{HT86}):
\begin{proposition}\label{global_wp}
   Suppose $\la >0$, $(\phi, \, \psi)$ satisfies \eqref{inital_datum}-\eqref{weighted} with $s \geq 2$. Then  there exists a unique global solution $g_0 \in L^\infty( [0, \infty), H^s(\Rr^2) )$ to \eqref{ls_g_0}
   with the following estimate
\begin{align*}
     \|g_0(t, \cdot)\|_{H^s(\Rr^2)} \leq C_1(\phi, \psi, s), \quad \forall \, t \in [0, \infty). 
\end{align*}
Moreover, the following decay estimate holds
\begin{align*}
     \| g_0(t, \cdot) \|_{L^\infty(\Rr^2)} \leq C_2(\phi, \psi)(1 + t)^{-1}, \quad \forall \, t  \in (0, \infty).
\end{align*}
\end{proposition}
As pointed out in \cite{FZ21},  the prefactor $C_2(\phi, \psi)$ in decay estimate depends not only on the specific norms of $\phi$ and $\psi$, but also on their profiles.

 Our first main result shows that the error estimates in the non-relativistic limit are of order $O(\e^2)$ with a prefactor growing algebraically in time $t$ up to long time interval:
\begin{theorem} \label{defo_1}
    Suppose $\la > 0$, $(\phi, \psi)$ satisfies \eqref{inital_datum}-\eqref{weighted} with $s  > 9$. Then the Cauchy problem of the cubic Klein-Gordon equation \eqref{K_L}--\eqref{inital_KL} admits a unique solution
    \begin{align*}
        u \in C([0, T_\e], H^{s - 8}(\Rr^2)),\quad   T_\e := \frac{T_0}{\e^{\a}}, \quad \a := \frac{2}{N_{\phi,\, \psi} + 1},
    \end{align*}
    with $T_0>0$  independent of $\e$.  Here $N_{\phi,\psi} = \max \{\hat{N} _{\phi,\psi},\, \tilde{N} _{\phi,\psi} \}$ and $\hat{N}_{\phi,\psi}$, $ \tilde{N}_{\phi,\psi}$ are defined in \eqref{definehatN} and \eqref{define_tildeN} with $K=2$. Moreover, the following error estimate holds
    \begin{align*} %\label{defo_1_estimate}
         \|u - ( e^{i\xth} g_0 + e^{-i\xth} \xwl{g}_0)(t, \cdot)\|_{H^{s-8}(\Rr^2)} \leq C_{\phi,\psi}(1 + t)^{N_{\phi,\psi}}  \e^2, \quad t \leq \frac{T_0}{\e^{\a}},
    \end{align*}
     with $\xth = t\e^{-2}$ here and in the sequel. 
\end{theorem}
We also show that the convergence rate is of order $O(\varepsilon)$ for initial data with limited regularity.
\begin{theorem} \label{defo_2}
    Suppose $ \la > 0 $, $(\phi, \psi)$ satisfies \eqref{inital_datum}-\eqref{weighted} with $s  > 5$. Then the Cauchy problem of the cubic Klein-Gordon equation \eqref{K_L}--\eqref{inital_KL} admits a unique solution
\begin{align*}
    u \in C([0, T_\e], H^{s-4}(\mathbb{R}^2)), \quad
    T_\e := \frac{T_1}{\epsilon^{\frac{\alpha}{2}}}, \quad \alpha := \min\{1, \frac{2}{\tilde{N}_{\phi,\psi} + 1}\},
\end{align*} %一般都是不需要空格的 $$里面，要有统一的习惯
with $T_1>0$ independent of $\e$. Here $ \tilde{N}_{\phi,\psi}$ is defined in  \eqref{define_tildeN} with $K = 0$.  Moreover, the following error estimate holds
\begin{align*}
    \|u - (e^{i\xth}g_0 + e^{-i\xth}\bar{g}_0)(t, \cdot)\|_{H^{s-4}(\mathbb{R}^2)} \leq C_{\phi,\psi} \left( (1 + t) + (1 + t) ^ {\tilde{N}_{\phi,\psi}} \right) \epsilon, \quad t \leq \frac{T_1}{\epsilon^{\frac{\alpha}{2}}},
\end{align*} %%\leq 写成\leq 就行了
 
\end{theorem}
We can actually show the validity of the Schr\"odinger approximation up to an arbitrary order with sufficiently smooth initial data. For example, the approximation of order $O(\e^4)$ was already observed in the numerical studies in  \cite{SZ20} and \cite{BZ19_2}. This constitutes the following theorem.
\begin{theorem} \label{defo_3}
   Suppose $\lambda > 0$, $K \in \mathbb{Z}_{+} \cap \mathbb{Z}_{\rm even}$ and $(\phi, \psi)$ satisfies \eqref{inital_datum}-\eqref{weighted} with $s > 2K + 5$. Then the Cauchy problem of the cubic Klein-Gordon equation \eqref{K_L}-\eqref{inital_KL} admits a unique solution 
   \begin{align*}
     u \in C([0, T_\varepsilon], H^{s - 2K - 4}(\mathbb{R}^2)), \quad T_\varepsilon &:= \frac{T_2}{\e^{\a}}, \quad \alpha := \frac{2}{N_{\phi,\psi} + 1},
   \end{align*}
   with $T_2 > 0$ independent of $\e$. Here $N_{\phi,\psi} = \max \{\hat{N} _{\phi,\psi},\, \tilde{N} _{\phi,\psi} \}$ and $\hat{N}_{\phi,\psi}$, $ \tilde{N}_{\phi,\psi}$ are defined in \eqref{definehatN} and \eqref{define_tildeN}. Moreover, the following error estimate holds
    \begin{align} \label{defo_3_estimate1}
    \begin{split}
        &\|(u - u_a)(t, \cdot)\|_{H^{s-2K - 4}(\mathbb{R}^2)}  
        \leq  C_{\phi,\psi}(1 + t)^{ \frac{K}{2} N_{\phi,\psi}+\frac{K}{2}} \e^{K+1}, \quad t \leq \frac{T_2}{\e^{\a}},
    \end{split}
    \end{align} 
    where $u_a$ is the approximate solution having the  form
    \begin{align} \label{form_Ua}
        u_a = u_0 + \e^2 u_2 + \cdots + \e^{K} u_{K} + \e^{K+2} u_{K+2},
    \end{align}
    with
   \begin{align} \label{form_Ua1}
       u_0 := e^{i\xth}g_0 + e^{-i\xth}\bar{g}_0, \quad u_n = \sum_{|p| \leq n+1} e^{i p \theta} u_{n,p}.
   \end{align}
   In addition, for all $n = 2, 4, \ldots, K$, $|p| \leq n + 1$ and $t \in (0, \infty)$, the following estimates hold
   \begin{align} \label{defo_3_estimate2}
   \begin{split}
        &\|g_0(t, \cdot)\|_{H^{s}} \leq C_{\phi,\psi}, \\%%s_a 不是sa吧， 我不确定你自己确认一下，
      &\|u_{n,p}(t, \cdot)\|_{H^{s-2n}} \leq C_{\phi,\psi} (1 + t)^{\frac{n}{2} \hat{N}_{\phi,\psi} + \frac{n-2}{2}}, \\
       &\|u_{K+2,p}(t, \cdot)\|_{H^{s-2K-2}} \leq C_{\phi,\psi} (1 + t)^{\frac{K}{2} \hat{N}_{\phi,\psi} + \frac{K-2}{2}}.
    \end{split}
    \end{align}
 
 \end{theorem}

Similarly, we have the following results for the focusing case ($\la < 0$).
\begin{theorem} \label{fo_1}
    Suppose $\la < 0$, $K \in \mathbb{Z}_{+} \cap \mathbb{Z}_{\rm even}$ and $(\phi, \psi) $ satisfies \eqref{inital_datum} with $s > 2K + 5$. Let $ T^*$ be the existence time to the Cauchy problem of the cubic Schr\"{o}dinger equation \eqref{ls_g_0}. Then the Cauchy problem of the cubic Klein-Gordon equation \eqref{K_L}-\eqref{inital_KL}admits a unique solution
    \begin{align*}
        u \in C([0, T^*_\e), H^{s-2K-4}(\mathbb{R}^2)), 
    \end{align*}
    where the existence time $T^*_\e$ satisfies
    \begin{align*}  %%其实我习惯都是用label的，你如果都是tag，也就全部用tag，用label就全用label
    \liminf_{\e \to 0} T^*_\e \geq T^*.
    \end{align*}
    Moreover, the following error estimate holds
    \begin{align*} %%公式前面不需要用一个quad的 sa的地方都需要去改
     \|(u - u_a)(t, \cdot)\|_{H^{s-2K-4}(\mathbb{R}^2)} \leq C_{\phi,\psi}(t)\e^{K + 1}, \quad t < \min\{T^*, T^*_\e\},
    \end{align*} %% 这个地方肯定不能超过区域.
    where $u_a$ has the same form as \eqref{form_Ua} and \eqref{form_Ua1}. For all $n = 2, \,4,\, \cdots,\, K+2$, $|p| \leq n + 1$ and $t < \min\{T^*, T^*_\e\}$, the following estimates hold
    \begin{align*}
      \|g_0(t, \cdot)\|_{H^{s}} \leq C_{\phi,\psi}(t), \quad \|u_{n,p}(t, \cdot)\|_{H^{s-2n}} \leq C_{\phi,\psi}(t).
    \end{align*}
    In particular,
    \begin{align*} 
      \|u - (e^{i\xth}g_0 + e^{-i\xth}\bar{g}_0 ) \|_{H^{s-2K-4}(\mathbb{R}^2)} \leq C_{\phi,\psi}(t)\e^2, \quad t < \min \{T^*, T^*_\e\}.
    \end{align*} %%这里肯定是错的，少了括号 
\end{theorem}

The rest of the paper is devoted to proving the above theorems. We will reformulate non-relativistic limit problems into stability problems of the WKB approximate solutions in geometric optics, as detailed in Section \ref{reformulation}. The error estimates above can be seen as corollaries of the stability results established in Section  \ref{reformulation}, as detailed in Section \ref{error}.

\section{Reformulation}  \label{reformulation}

\subsection{The equivalent symmetric hyperbolic system}
We will rewrite the cubic Klein-Gordon equation as a symmetric hyperbolic system. Specifically, we define
\begin{align}   \label{define_U}
    U := (w^{\rm T},\, v,\, u)^{\rm T} := ( \e \pa_{x_1} u,\, \e \pa_{x_2} u,\, \e^2 \pa_t u,\, u )^{\rm T},
\end{align} 
where $w := \e \nabla u = \e (\pa_{x_1} u,\, \pa_{x_2} u)^{\rm T}, v := \e^2 \pa_t u$.   Then, equation \eqref{K_L} is equivalent to
\begin{align}  \label{K_L_change}
    \pa_t U - \frac{1}{\e} A(\pa_x) U + \frac{1}{\e^2} A_0 U = F(U),
\end{align}
with
\begin{align*}
    A(\pa_x) = \begin{pmatrix}
    0_{2 \times 2} & \nabla & 0_2 \\
    \nabla^{\rm T} & 0 & 0 \\
    0_2^{\rm T} & 0 & 0
    \end{pmatrix}, \quad
    A_0 = \begin{pmatrix}
    0_{2 \times 2} & 0_2 & 0_2 \\
    0_2^{\rm T} & 0 & 1 \\
    0_2^{\rm T} & -1 & 0
    \end{pmatrix}, \quad F(U) = -\begin{pmatrix}
    0_2 \\
    f(u) \\
    0
    \end{pmatrix} . %%都不要用tag
\end{align*}
Here $0_{2 \times 2}$ denotes zero matrix of order $2 \times 2$ and $0_2$ denotes  the zero column vector of dimension $2$. If there is no confusion from the context, we may omit the subscript $2$ or $ 2 \times 2 $ and simply use $0$. From \eqref{inital_KL}, we naturally get the initial datum of $U$
\begin{align} \label{inital_U}
    U(0,\, \cdot) = ( \e \nabla^{\rm T} \phi,\, \psi,\, \phi )^{\rm T}. 
\end{align}

\subsection{Main Results} \label{idea}

We will employ the classical WKB expansion technique from geometric optics to construct approximate solutions to \eqref{K_L_change} and then study the stability of such WKB solutions. The core idea of the WKB method involves expanding the solution as a formal power series in the small parameter $\epsilon$, where each term in the series is a trigonometric polynomial in $\theta$:
\begin{align} \label{WKB_inital}
    U_a = \sum_{n=0}^{K+2} \e^n U_n, \quad U_n = \sum_{p \in H_n} e^{i p \xth} U_{n,p}, \quad K \in \Zz_{\geq 0}, \quad H_n \subset \Zz.
\end{align}
Analogous to  \eqref{define_U}, we use the notation
\begin{align*}
    U_a = \begin{pmatrix} 
    w_a \\ 
    v_a \\ 
    u_a 
    \end{pmatrix}, \quad U_n = \begin{pmatrix} 
    w_n \\ 
    v_n \\ 
    u_n 
    \end{pmatrix}, \quad U_{n,p} = \begin{pmatrix} 
    w_{n,p} \\ 
    v_{n,p} \\ 
    u_{n,p} 
    \end{pmatrix}.
\end{align*}
In \eqref{WKB_inital}, $H_n$ denotes the $n$-th order harmonics set which will be given later. 
%The leading harmonic $H_0$ is defined as $H_0 := \{p \in \Zz : \det(i p  + A_0) = 0\}$. For the non-homogeneous case ($A_0 \neq 0$), the set $H_0$ is always finite.  Higher-order harmonics are determined by the fundamental harmonics and the nonlinearities of the system. In general, $H_n \subset H_{n+1}$ and the structure of $H_{n+1}$ depends on both $H_n$ and the nonlinearities in the equation. 
This will be elaborated in Section \ref{WKB}.

We substitute \eqref{WKB_inital} into \eqref{K_L_change} and derive the equation of order $O(\e^n)$:
\begin{align} \label{Phi_np}
    \Phi_{n,p} := \pa_t U_{n,p} - A(\pa_x) U_{n + 1,p} + (i p I_4 + A_0) U_{n + 2,p} - F(U_a)_{n,p} = 0,
\end{align}
where $n \in \mathbb{Z}_{\geq -2}$, $p \in \mathbb{Z}$ and 
\begin{align} \label{form_ff}
    F(U_a)_{n, p} := \begin{pmatrix} 
    0_2 \\ 
    -f(u_a)_{n,p} \\ 
    0 
    \end{pmatrix}, \quad f(u_a)_{n, p} := \la \sum_{n_1 + n_2 + n_3 = n} \sum_{p_1 + p_2 + p_3 = p} u_{n_1,p_1} u_{n_2,p_2} u_{n_3,p_3}.
\end{align}
For notational consistency,  we impose $U_n = 0$ for $n = -2,\, -1$. The idea to construct approximate solutions is to solve $\Phi_{n,p} = 0$ up to some non-negative order $K$, so that $U_a$ solves \eqref{K_L_change} approximately with an error of order $O(\e^{K + 1})$. We will choose $K \in \mathbb{Z}_{\rm even}$ denoting the even integers for the convenience of the statements of main results. 

We will demonstrate that WKB approximate solutions of arbitrary order can be constructed. We first consider the focusing case with $\la < 0$.
\begin{theorem} \label{f_remainder_estimate}
    Suppose $\la < 0$, $K \in \Zz_{+} \cap \Zz_{\rm even} $ and $(\phi, \psi)$ satisfies \eqref{inital_datum} with $s > 2 K + 5$. Let $T^*$ be the existence time to the Cauchy problem of the cubic Schr\"{o}dinger equation \eqref{ls_g_0}. Then there exists a WKB solution $U_a$ of the form \eqref{WKB_inital} with  $U_{n,p} \in C([0, T^*); H^{s-2n}(\Rr^2) )$  for all $n = 0,\, 1,\, 2,\, \ldots,\, K + 2$ and $p \in H_n$ given in \eqref{define_Hn}. Moreover, for all $t \in (0, T^*)$, $U_a$ satisfies
\begin{align} \label{eq_Ua} 
\begin{split}
\begin{cases}
    &\pa_t U_a - \frac{1}{\e} A(\pa_x) U_a + \frac{1}{\e^2} A_0 U_a = F(U_a) - \e^{K+1} R_\e,\\
    &U_a(0, \cdot) = U(0, \cdot) - \e^{K+2} r_\e,
\end{cases}
\end{split} 
\end{align} 
where
\begin{align*} 
    \|R_\e(t, \cdot)\|_{H^{s-2K-4}} \leq C_{\phi,\psi}(t), \quad \|r_\e\|_{H^{s-2K-2}} \leq C_{\phi,\psi}. 
\end{align*}
\end{theorem}
While for the defocusing case $\la>0$, we can construct a global WKB solution.
\begin{theorem} \label{d_remainder_estimate}
    Suppose $\la > 0$, $K \in \Zz_{+} \cap \Zz_{\rm even} $ and $(\phi, \psi)$ satisfies \eqref{inital_datum}-- \eqref{weighted} with $s > 2 K + 5$. Then there exists a WKB solution $U_a$ of the form \eqref{WKB_inital} with $U_{n,p} \in C([0, \infty); H^{s-2n}(\mathbb{R}^2))$  for all $n = 0,\, 1,\, 2,\, \ldots,\, K + 2$ and $p \in H_n$ given in \eqref{define_Hn}. Moreover, for all $t \in (0, \infty)$, $U_a$ satisfies \eqref{eq_Ua} with 
    \begin{align*} 
        \|R_\e(t, \cdot)\|_{H^{s-2K-4}} \leq C_{\phi,\psi} \left( 1  + (1 + t)^{\frac{K}{2} \hat{N}_{\phi,\psi} + \frac{K-2}{2}} \right),  
        \quad \|r_\e\|_{H^{s-2K-2}} \leq C_{\phi,\psi} .
    \end{align*}
    In particular, $U_{n,p}$ satisfies all the properties in Corollary \ref{existence}. The constant $\hat{N}_{\phi,\psi}$ is defined in \eqref{definehatN}. 
\end{theorem}

We will show that the WKB solutions given in Theorems \ref{f_remainder_estimate} and \ref{d_remainder_estimate} are stable. For the focusing case, we have:
\begin{theorem} \label{f_ex}
    Suppose $\la<0$, $K \in \Zz_{+} \cap \mathbb{Z}_{\rm even}$ and $(\phi, \psi)$ satisfies \eqref{inital_datum} with $s > 2K + 5$. Let $T^*$ be the existence time to the Cauchy problem of the cubic Schr\"{o}dinger equation \eqref{ls_g_0}. Then the Cauchy problem of the cubic Klein-Gordon equation \eqref{K_L_change}--\eqref{WKB_inital} admits a unique solution $U \in C([0, T^*_\e)$;$ H^{s-2K-4}(\Rr^2))$ where the existence time satisfies
    \begin{align*}
        \liminf_{\e \to 0} T^*_\e \geq T^*.
    \end{align*} %%\geq \geq 是一样的,可改可不改 包括\mathbb{R} 和 \Rr 一样的, 改不改没啥问题
    Moreover, the WKB solution $U_a$ given in Theorem \ref{f_remainder_estimate} is stable in the following sense:
    \begin{align*}
        \| (U - U_a)(t, \cdot) \|_{H^{s-2K-4}(\Rr^2)} \leq C_{\phi,\psi}(t) \e^{K+1}, \quad t < \min \{ T^*, T_{\e}^*\}.
    \end{align*} %% 里面的d也不要出现
\end{theorem}
For the defocusing case, we can show the stability over long time. For data with high regularity, we have the following result:
\begin{theorem} \label{d_ex_h}
    Suppose $\la > 0$, $K \in \Zz_+ \cap \Zz_{\rm even} $ and $(\phi, \psi)$ satisfies \eqref{inital_datum} and \eqref{weighted} with $s > 2 K + 5$. Then the Cauchy problem of the cubic Klein-Gordon equation \eqref{K_L_change}-\eqref{WKB_inital} admits a unique solution $U \in C([0, T_\e]; H^{s-2K-4}(\mathbb{R}^2))$ with
\begin{align*}
T_\e := 
\frac{T_0}{\e^{\a}}, \quad \alpha := \frac{2}{N_{\phi,\psi} + 1},
\end{align*}
where $T_0>0$ is independent of $\e$. Moreover, the WKB solution $U_a$ given in Theorem \ref{d_remainder_estimate} is stable: 
\begin{align*}
    \| (U - U_a)(t, \cdot) \|_{H^{s-2K-4}(\mathbb{R}^2)} \leq C_{\phi,\psi} (1 + t)^{\frac{K}{2} N_{\phi,\psi} + \frac{K}{2}}  \e^{K+1}, \quad \forall\,t \leq T_0 \e^{-\alpha}. 
\end{align*}
The constant $N_{\phi,\psi} = \max\{\hat{N}_{\phi, \psi}, \tilde{N}_{\phi,\psi}\}$ where $\hat{N}_{\phi,\psi}$ and $\tilde{N}_{\phi,\psi}$ are defined in \eqref{definehatN} and \eqref{define_tildeN} respectively. 
\end{theorem}
For the case with limited regularity ($K = 0$), we show the following result.
\begin{theorem} \label{d_ex_l}
    Suppose $\la > 0$ and $(\phi, \psi)$ satisfies \eqref{inital_datum} --\eqref{weighted}  with $s >  5$. Then the Cauchy problem of the cubic Klein-Gordon equation \eqref{K_L_change}-\eqref{WKB_inital} admits a unique solution $U \in C([0, T_\e]; H^{s-4}(\mathbb{R}^2))$ with
\begin{align*}
    \hat{T}_\e := 
    \frac{T_1}{\e^{\frac{\a}{2}}}, \quad \alpha := \min \{1, \frac{2}{\tilde{N}_{\phi,\psi} + 1}\} ,
\end{align*}
where $T_1>0$ is independent of $\e$. Moreover, the WKB solution $U_a$ given in Theorem \ref{d_remainder_estimate} is stable: 
\begin{align*}
    \| (U - U_a)(t, \cdot) \|_{H^{s-4}(\mathbb{R}^2)} \leq C_{\phi,\psi} \left( (1 + t) + (1 + t)^{\tilde{N}_{\phi,\psi}}  \right) \e, \quad \forall\,  t \leq \frac{T_1}{\e^{\frac{\a}{2}}}.
\end{align*}
The constant $\tilde{N}_{\phi,\psi}$ is defined in  \eqref{define_tildeN} with $K = 0$. 
\end{theorem}

The forthcoming Section \ref{WKB} is devoted to the proofs of Theorems \ref{f_remainder_estimate} and \ref{d_remainder_estimate}.  We will finally prove the stability results in Theorems \ref{f_ex}, \ref{d_ex_h} and \ref{d_ex_l} in Section \ref{sec-stability}. Our main results in Theorems \ref{defo_1}, \ref{defo_2}, \ref{defo_3}, \ref{fo_1} can be seen as corollaries of Theorems \ref{f_ex}, \ref{d_ex_h} and \ref{d_ex_l}.

% \cp{remark}

\section{WKB expansion and approximate solutions} \label{WKB}
In this section, we use WKB expansion introduced in Section \ref{reformulation} to construct approximate solutions. 
\subsection{$O(\epsilon^{-2})$}
As shown in Section \ref{idea}, after substituting the formal expansion \eqref{WKB_inital} into \eqref{K_L_change}, the equation containing all terms of order $O(\e^{-2}$) is $\Phi_{-2,p} = 0$:
\begin{align} \label{eq_phi-2p}
    L_p U_{0,p} = 0, \quad \forall\, p \in \mathbb{Z}.
\end{align}
%This is equivalent to
%\begin{align} \label{PipU0p}
%    \Pi_p U_{0,p} = U_{0,p}, \quad \text{for any } p \in \mathbb{Z}.
%\end{align}
We denote $L_p = ipI_4 + A_0$. % $\Pi_p$ in \eqref{PipU0p} is the orthogonal projection onto the kernel of $L_p$, such that
%\begin{align*} 
 %   L_p \Pi_p = 0,  \quad \text{for any } p \in \Zz. 
%\end{align*}
If $L_p$ is invertible, let $L_p^{-1}$ be the inverse of $L_p$. %Otherwise, let $L_p^{-1}$ be the partial inverse of $L_p$的, such that
%\begin{align*}
   % L_p^{-1} \Pi_p = 0, \quad L_p^{-1} L_p = L_p L_p^{-1} = I_4 - \Pi_p,  \quad \text{for any } p \in \mathbb{Z}.
%\end{align*}
Since $L_p$ is invertible for $|p| \geq 2$, we can directly get the form of $L_p^{-1}$
\begin{align} \label{Lp-1form}
    L_p^{-1} = 
    \begin{pmatrix}
      -\frac{i}{p}I_2 & 0 & 0 \\
      0 & \frac{ip}{1-p^2} & \frac{-1}{1-p^2} \\
      0 & \frac{1}{1-p^2} & \frac{i p}{1-p^2}
    \end{pmatrix}, \quad \forall\, |p| \geq 2.
\end{align}
Then we can deduce from \eqref{eq_phi-2p} that
\begin{align} \label{form_U0p}
U_{0,p} = 0, \quad \forall\,  |p| \geq 2 .
\end{align}
We only need to consider $p = -1,\, 0,\, 1$. By calculation,
\begin{align} \label{kerL}
\begin{split}
    &\ker L_0 = (w, 0, 0), \quad  \forall\, w \in \mathbb{R}^2,  \\
    &\ker L_1 = \text{span}\{e^+\}, \quad \ker L_{-1} = \text{span}\{e^-\}.
\end{split}
\end{align}
The conjugate couple $e^+$ and $e^-$ are defined as
\begin{align*} 
    e^{\pm} := (0,\, \pm i,\, 1)^{\rm T}.
\end{align*}
%As a result, the orthogonal projections $\Pi_{\pm 1}$ are defined as
%\begin{align} \label{op} %%特别是forall后面都是需要一个\, 的, 不然就是间隔太近了, 不好看 排版的最终目的就是美观, 一致性.
  % \Pi_{\pm 1} V = \frac{1}{|e^{\pm}|^2} (V, e^{\pm}) e^{\pm} = \frac{1}{2} (V, e^{\pm}) e^{\pm}, \quad \forall \, V \in \mathbb{R}^{4},
%\end{align}
%Here $(U, V)$ denotes the inner product in $\mathbb{C}^{4}$, $(U, V) = U \cdot \bar{V}$. 
This implies
\begin{align} \label{form_U01}
\begin{split}
        &U_{0,1} = g_0 e^+, \quad  U_{0,-1}  = G_0 e^-,  \quad g_0 \text{ and } G_0 \text{ are scalar functions to be determined}.
\end{split}    
\end{align}
To ensure the reality of the solution, it is natural to impose
\begin{align} \label{reality}
    U_{0,-1} = \bar{U}_{0,1}.
\end{align}
Thus, we have $G_0 = \bar{g}_0$.
\begin{remark} \label{conjugate}
To ensure the reality of the WKB solutions, we shall always require 
    \begin{align*}
    U_{n,p} = \bar{U}_{n,-p}, \quad \forall\, n, p \in \mathbb{Z}.
    \end{align*}
    As a result, we only consider $p \geq 0$ in the sequel. 
\end{remark}
%Moreover, by direct computation, we have
%\begin{align} \label{P0P1}
 %   \Pi_0 = \begin{pmatrix}
  %      I_2 & 0_2 & 0_2 \\
   %     0_2^{\rm T} & 0 & 0 \\
  %      0_2^{\rm T} & 0 & 0
  %  \end{pmatrix}, \quad \Pi_1 = \frac{1}{2} \begin{pmatrix}
   %     0_{2 \times 2} & 0_2 & 0_2 \\
   %     0_2^{\rm T} & 1 & i \\
   %     0_2^{\rm T} & -i & 1
  %  \end{pmatrix}, \quad \Pi_{-1} =\bar{\Pi}_1,
%\end{align}
%Combining with the fact that $\Pi_p^2 = \Pi_p$, we  have
%\begin{align} \label{PipLp}
 %    \Pi_p L_p = 0 = L_p \Pi_p.
%\end{align}
%\begin{align*}
 %    \Pi_p L_p^{-1} = 0 = \Pi_p L_p^{-1}.
%\end{align*}
%Also, we have 
%\begin{align} \label{L0L1}
%    L_0^{-1} = 
%    \begin{pmatrix}
 %     0_{2 \times 2} & 0_2 & 0_2 \\
 %     0_2^{\rm T} & 0 & -1 \\
 %     0_2^{\rm T} & 1 & 0
 %   \end{pmatrix}, \quad L_1^{-1} = 
  %  \begin{pmatrix}
  %    -i I_2 & 0_2 & 0_2 \\
 %     0_2^{\rm T} & -\frac{i}{4} & -\frac{1}{4} \\
 %     0_2^{\rm T} & \frac{1}{4} & -\frac{i}{4}
%\end{pmatrix}, \quad L_{-1}^{-1} = \bar{L_1}^{-1}.
%\end{align}
%Then, we can deduce from \eqref{eq_phi-2p} and the invertibility of $L_p$ that

For simplicity, we suppose $U_{0,0} = 0$. Then the leading term $U_0$ has the form
\begin{align} \label {U0form}
    U_0 = e^{i\theta} g_0 e^+ + e^{-i\theta} \bar{g}_0 e^-,
\end{align}
where $g_0$ is a scalar function.

\subsection{$O(\epsilon^{-1})$} \label{o-1}

The equation containing all terms of order $O(\e^{-1}$) is $\Phi_{-1,p} = 0$:
\begin{align} \label{eq-1p}
    -A(\pa_x) U_{0,p} + L_p U_{1,p} = 0, \quad \forall\, p \in \Zz.
\end{align}

\medskip

%As shown in Proposition \ref{Ueven} and Remark \ref{Ueven_all}, it is sufficient to consider \eqref{eq-1p} when p is odd. 
For $p = 1$, applying $L_1^{-1}$ to \eqref{eq-1p} implies
\begin{align} \label{c_eq_Phi-1p_1}
    (I_4 - \Pi_1) U_{1,1} = L_1^{-1} A(\pa_x) U_{0,1}.
\end{align}
Here $L_1^{-1}$ is the partial inverse of $L_1$ and $\Pi_1$ is the orthogonal projection onto the kernel of $L_1$ such that
\begin{align*}
    & L_1 \Pi_1 = \Pi_1 L_1 = 0,
     \quad L_1^{-1} L_1 = L_1 L_1^{-1} = I_4 - \Pi_1. 
\end{align*}
Direct computation gives that
\begin{align} \label{f_L1-1}
   \Pi_1 = \frac{1}{2} \begin{pmatrix}
       0_{2 \times 2} & 0 & 0 \\
       0 & 1 & i \\
        0 & -i & 1
   \end{pmatrix}, \quad   L_1^{-1} = 
    \begin{pmatrix}
      -i I_2 & 0 & 0 \\
      0 & -\frac{i}{4} & -\frac{1}{4} \\
      0 & \frac{1}{4} & -\frac{i}{4}
\end{pmatrix}. 
\end{align}
Moreover, by \eqref{eq_phi-2p} and the definition of $\Pi_1$, there holds
\begin{align*}
    \Pi_1 U_{0,1} = U_{0,1}.
\end{align*}
Together with \eqref{form_U01} and \eqref{f_L1-1}, we can deduce from \eqref{c_eq_Phi-1p_1} that
\begin{align} \label{c_eq_Phi-1p}
    (I_4 - \Pi_1) U_{1,1} = L_1^{-1} A(\pa_x) U_{0,1} = L_1^{-1} A(\pa_x) \Pi_1 U_{0,1} = L_1^{-1} \begin{pmatrix} 
      i \nabla g_0 \\ 
      0 \\ 
      0 
    \end{pmatrix} = \begin{pmatrix} 
      \nabla g_0 \\ 
      0 \\ 
      0 
    \end{pmatrix}.
\end{align}
By \eqref{kerL} and the definition of $\Pi_1$, we know that
\begin{align*}
    \Pi_1 U_{1,1} = g_1 e^+, \quad g_1 \text{ is a scalar function to be determined}.
\end{align*}

\medskip

For $p \geq 2$, we know that $L_p$ is invertible. We  use \eqref{form_U0p} and $U_{0,0} = 0$ to derive
\begin{align} \label{U1p=0}
    U_{1,p} = 0, \quad \forall\, p  \geq 2.
\end{align}

Since $U_{0,0}=0$, we have $L_1 U_{1,0} = 0$. Then, for $p=0$, We can choose $U_{1,0}=0$ as well.

Thus, the form of $U_1$ is
\begin{align} \label{U1form}
    U_1 = e^{i\theta} \left[ g_1 e^+ + \begin{pmatrix} 
    \nabla g_0 \\ 
    0 \\ 
    0 
    \end{pmatrix} \right] + \text{c.c.},
\end{align}
where c.c. denotes the related complex conjugate.

\subsection{$O(\epsilon^0)$}

The equation containing all terms of order $O(\e^0)$ is $\Phi_{0,p} = 0$:
\begin{align} \label{eq_Phi0p}
    \pa_t U_{0,p} - A(\pa_x) U_{1,p} + L_p U_{2,p} = F(U)_{0,p}, \quad \forall\, p \in \Zz,
\end{align}
where  $F(U)_{0,p}$ takes the form in \eqref{form_ff} with  %%寒假的时候多改改, 最好能改完
\begin{align} \label{form_fu0}
\begin{split}
    f(u_0) &= \la (e^{i\xth} g_0 + e^{-i\xth} \bar{g}_0)^3
    = \la ( e^{3i\theta} g_0^3 + 3e^{i\theta} |g_0|^2 g_0 + 3e^{-i\theta} |g_0|^2 \bar{g}_0 + e^{-3i\theta} \bar{g}_0^3 ).
\end{split}
\end{align}

\medskip

Substituting $p = 1$ into equation \eqref{eq_Phi0p}, we actually get
\begin{align} \label{c_eq_Phi0p}
    \pa_t U_{0,1} - A(\pa_x) U_{1,1} + L_1 U_{2,1} = F(U_0)_1 = 
    \begin{pmatrix} 
      0_2 \\
      -3\lambda |g_0|^2 g_0 \\ 
      0 
    \end{pmatrix}.
\end{align}
Applying $\Pi_1$ to \eqref{c_eq_Phi0p} gives
\begin{align} \label{P1U01}
    \pa_t \Pi_1 U_{0,1} - \Pi_1 A(\pa_x) U_{1,1} = \Pi_1 \begin{pmatrix} 
    0_2 \\
    -3\la |g_0|^2 g_0 \\ 
    0 
    \end{pmatrix}.
\end{align}
Use \eqref{c_eq_Phi-1p} and we can get the decomposition %%这里的公式肯定是不行的，明显超过了
\begin{align} \label{decompositon}
\begin{split}
     \Pi_1 A(\pa_x) U_{1,1} &= \Pi_1 A(\pa_x) (I_4 - \Pi_1) U_{1,1} + \Pi_1 A(\pa_x) \Pi_1 U_{1,1}  \\ 
    &=\Pi_1 A(\pa_x) L_1^{-1} A(\pa_x) \Pi_1 U_{0,1} + \Pi_1 A(\pa_x) \Pi_1 U_{1,1} \\
    &=-\frac{i}{2} \Delta_x \Pi_1 U_{0,1},
\end{split}
\end{align}
where we use the two formulas below:
\begin{align} \label{P1APx}
    \Pi_1 A(\pa_x) \Pi_1 = 0,  \quad 
    \Pi_1 A(\pa_x) L_1^{-1} A(\pa_x) \Pi_1 = -\frac{i}{2} \Delta_x \Pi_1,
\end{align}
 the proof of which is rather straightforward. Thus, by \eqref{P1U01}, \eqref{decompositon} and \eqref{P1APx}, we have
\begin{align*}
    \pa_t \Pi_1 U_{0,1} + \frac{i}{2} \Delta_x \Pi_1 U_{0,1} = \Pi_1 
    \begin{pmatrix} 
     0_2 \\ 
      -3 \la |g_0|^2 g_0 \\ 
      0 
    \end{pmatrix}.
\end{align*}
Combining with \eqref{form_U01},  we deduce
\begin{align*} 
    2i \pa_t g_0 - \Delta_x g_0 + 3\la |g_0|^2 g_0 = 0.
\end{align*}
This is exactly the Schr\"{o}dinger equation \eqref{ls_g_0}.

Recall the initial data $U(0, \cdot)$ in \eqref{inital_U}. We then choose the initial data of $g_0$ such that the difference between $U$ and $U_a$ can be as small as possible:
\begin{align*}
    g_0(0, \cdot) e^+ + \bar{g}_0(0, \cdot) e^- = (0, \psi, \phi)^{\rm T}.
    \end{align*}
This implies
\begin{align} \label{g0inital}
    g_0(0, \cdot) = \frac{\phi - i\psi}{2},
    \end{align}
which is exactly the initial data in \eqref{ls_g_0}. Together with \eqref{c_eq_Phi-1p} and \eqref{inital_U}, we can actually have
\begin{align} \label{U0_1}
    U(0, \cdot) = U_0(0, \cdot) + \e (1 - \Pi_1) U_{1,1}(0, \cdot) + \text{c.c.}.
\end{align}
 Combining with the fact  $L_1^{-1} U_{0,1} = 0$ and the decomposition of $U_{1,1}$ in Section \ref{o-1}, we apply $L_1^{-1}$ to \eqref{c_eq_Phi0p} to derive
\begin{align*} 
\begin{split}
    (I_4 - \Pi_1) U_{2,1} &= L_1^{-1} A(\pa_x) U_{1,1} + L_1^{-1} \begin{pmatrix} 
        0_2 \\ 
        -3\lambda |g_0|^2 g_0 \\ 
        0 
    \end{pmatrix} \\
    &= L_1^{-1} A(\pa_x) \Pi_1 U_{1,1} + L_1^{-1} A(\pa_x) L_1^{-1} A(\pa_x) \Pi_1 U_{0,1}  + L_1^{-1} \begin{pmatrix} 
        0_2 \\ 
        -3\la |g_0|^2 g_0 \\ 
        0 
    \end{pmatrix} .
\end{split}
\end{align*}
Noting that $\Pi_1 U_{1,1} = g_1 e^+ $ and $\Pi_1 U_{0,1} = g_0 e^+ $, we can simplify the formula above as
\begin{align} \label{f_1-P1U21}
\begin{split}
     (I_4 - \Pi_1) U_{2,1} &= \begin{pmatrix}
      \nabla g_1 \\ 
      0 \\ 
      0 
    \end{pmatrix} + \frac{1}{4} \begin{pmatrix} 
      0_2 \\ 
      -i(\Delta_x g_0 - 3\la |g_0|^2 g_0) \\ 
      \Delta_x g_0 - 3\lambda |g_0|^2 g_0 \end{pmatrix} = \begin{pmatrix} 
      \nabla g_1 \\ 
      0 \\ 
      0 
    \end{pmatrix} + \frac{1}{2} \begin{pmatrix} 
      0_2 \\ 
      \pa_t g_0 \\ 
      i \pa_t g_0 
    \end{pmatrix}.
\end{split}
\end{align}
Then there exists a scalar function $g_2$ such that
\begin{align} \label{p1U21}
\begin{split}
        \Pi_1 U_{2,1} = \left( g_2 - \frac{i}{2} \pa_t g_0 \right) e^+.
\end{split}
\end{align}
Hence
\begin{align*}
    U_{2,1} = g_2 e^+ + \begin{pmatrix} 
      \nabla g_1 \\ 
      \pa_t g_0 \\ 
      0 
    \end{pmatrix}.
\end{align*}

\medskip

For $p = 2$, it is straightforward to obtain that $F(U)_{0,2}=0$ in \eqref{eq_Phi0p}. So we can deduce from \eqref{form_U0p} and \eqref{U1p=0} that
\begin{align*}
    U_{2,2} = 0.
\end{align*}

\medskip

For $p = 3$, by  \eqref{form_U0p}, \eqref{U1p=0} \eqref{Lp-1form} and \eqref{form_fu0}, equation \eqref{eq_Phi0p} takes the form
\begin{align*}
    U_{2,3} = L_3^{-1} \begin{pmatrix} 0_2 \\ -\lambda g_0^3 \\ 0  \end{pmatrix} = \frac{\lambda g_0^3}{8} \begin{pmatrix} 0_2 \\ 3i \\ 1 \end{pmatrix}.
\end{align*}

For $p \geq 4$, we can directly deduce from the invertibility of  $L_p$ that
\begin{align*}
   U_{2,p} = 0, \quad \forall \, p \geq 4.
\end{align*}

For $p=0$, similar to the selection of $U_{1,0}$, we impose $U_{2,0}=0$.

To sum up,  the form of $U_2$ is
\begin{align} \label{U2form}
    U_2 = e^{i\theta} \left[ g_2 e^+ + \begin{pmatrix} \nabla g_1 \\ \pa_t g_0 \\ 0 \end{pmatrix} \right] + e^{3i\theta} \frac{\lambda g_0^3}{8} \begin{pmatrix} 0_2 \\ 3i \\ 1 \end{pmatrix} + \text{c.c.}.
\end{align}

In the previous arguments, we observe that
\begin{itemize}
    \item For all even integers $p$, we can choose $U_{n,p} = 0$;
    \item For $p$ large enough, $U_{n,p} = 0$.
\end{itemize}
These facts actually always hold, as shown in the following two propositions. 
\begin{proposition}
 \label{Ueven}
    Let $U_a$ be defined in \eqref{WKB_inital} and $U_0$ be given in \eqref{U0form}. If
    \begin{align} \label{Ueven=0}
      U_{n,p} = 0, \quad \forall\, p \in \mathbb{Z}_{\rm even} \, \text{ and } n = 0,\, \ldots,\, K + 2 ,
    \end{align}
    then the equation $\Phi_{n,p} = 0$ holds for all $n = -2,\, -1,\, 0,\, \ldots,\, K$ and $p \in \mathbb{Z}_{\rm even}$.
\end{proposition}
\begin{proof}
     By the assumption \eqref{Ueven=0}, for all $n = -2,\, -1,\, 0,\, \ldots,\, K$ and $p \in \mathbb{Z}_{\rm even}$, the left-hand side of \eqref{Phi_np} equals $0$.
    It is left to show the right-hand side is zero. Recall the form of $f(u_a)_{n,p}$ in \eqref{form_ff}.
    To ensure $p_{1} + p_{2} + p_{3} = p$, one of $p_{1}, p_{2}, p_{3}$ must be even. Due to  \eqref{Ueven=0}, we naturally have $f(u_a)_{n,p} = 0$.   
\end{proof}

\begin{proposition}\label{unplarge}
    For all $n = 0, \,1,\, 2, \ldots, K + 2$, let
    \begin{align*}
        p(n) := \begin{cases} 
          n + 1 & \text{if } n \in \Zz_{\rm even}, \\
          n & \text{if } n \in \Zz_{\rm odd}.
\end{cases}
\end{align*}
    Then, for all $p$ with $|p| \geq p(n)+1 $, we have $ U_{n,p} = 0 $. Hence, the order-n harmonic set is
    \begin{align} \label{define_Hn}
  H_n = \{ p \in \mathbb{Z}_{\rm odd} : |p| \leq p(n) \}.
    \end{align}
\end{proposition} 

\begin{proof}
    We prove this proposition by induction. We have shown the result with $n = 0, \,1,\, 2$ in previous sections. Assume that the result holds for $n = 0, \ldots, k+1$. Now we show that $U_{k+2,p} = 0$ for $|p| \geq p(k+2) + 1$. 
    Clearly, $p(k + 2) \geq p(k + 1) \geq p(k)$. Then, for $|p| \geq p(k + 2) + 1$,  we have $U_{k,p} = U_{k+1,p} = 0$ in the equation \eqref{Phi_np} and then
    \begin{align*}
        L_p U_{k+2,p} = F(U_a)_{k,p}.
    \end{align*}
    Recall the form of $f(u_a)_{k,p}$ in \eqref{form_ff}.
    If $k + 2$ is odd, then $k$ is odd and $p(k+2) = k+2, p(k) = k$. To ensure $ n_1 + n_2 + n_3 = k $, at least one of $\{n_1,\,n_2,\,n_3\}$ is odd. Without loss of generality, we assume $n_1$ is odd. By induction assumption, to ensure $f(u_a)_{k,p} \neq 0$, the following condition must hold
    \begin{align*}
        |p_1| \leq n_1, \quad |p_2| \leq n_2 + 1, \quad|p_3| \leq n_3 + 1.
    \end{align*}
    This ensures
    \begin{align*}
        |p| = |p_1 + p_2 + p_3| \leq |p_1| +|p_2| + |p_3| \leq n_1 + n_2 + n_3 + 2 \leq k + 2.
    \end{align*}
    Hence, for all $p$ with $|p| \geq p(k+2) + 1 = k + 3$, $f(u_a)_{k,p} = 0$. Thus $U_{k+2,p} = 0$.
    
    The case $k \in \Zz_{\rm even}$ can be shown similarly. The proof is completed.
\end{proof}
With Propositions \ref{Ueven} and \ref{unplarge}, it is sufficient to consider $p \in \Zz_{\rm even}$ smaller than $p(n)$ in the sequel.
\subsection{$O(\epsilon^1)$} \label{O1}

The equation containing all terms of order $O(\e^1)$ is $\Phi_{1,p} = 0$:
\begin{align} \label{phi_1p}
    \pa_t U_{1,p} - A(\pa_x) U_{2,p} + L_p U_{3,p} &= F_{1,p}, \quad \forall\,p \in \mathbb{Z}.
\end{align}
By \eqref{U0form} and \eqref{U1form}, we have 
\begin{align*}
    f_1 &= 3\lambda u_0^2 u_1 = 3\lambda (e^{i\theta} g_0 + e^{-i\theta} \bar{g}_0)^2 (e^{i\theta} g_1 + e^{-i\theta} \bar{g}_1) \\
    &= 3\lambda \left( e^{3i\theta} g_0^2 g_1 + e^{i\theta} (g_0^2 \bar{g}_1 + 2 |g_0|^2 g_1) + e^{-i\theta} (\bar{g}_0^2 g_1 + 2 |g_0|^2 \bar{g}_1) + e^{-3i\theta} \bar{g}_0^2 \bar{g}_1 \right).
\end{align*}

For $p = 1$, equation \eqref{phi_1p} takes the form
\begin{align} \label{U11_eq}
    \pa_t U_{1,1} - A(\pa_x) U_{2,1} + L_1 U_{3,1} = \begin{pmatrix} 
      0_2 \\ 
      -3\lambda (g_0^2 \bar{g}_1 + 2 |g_0|^2 g_1) 
      \\ 0 
    \end{pmatrix}.
\end{align}
Applying $\Pi_1$ to \eqref{U11_eq}, together with \eqref{P1APx} and \eqref{f_1-P1U21}, we can get
\begin{align} \label{P1U11}
\begin{split}
    \pa_t \Pi_1 U_{1,1} - \Pi_1 A(\pa_x) 
    \begin{pmatrix} 
      \nabla g_1 \\ 
      0 \\ 
      0 
    \end{pmatrix} - \frac{1}{2} \Pi_1 A(\pa_x) \begin{pmatrix} 
      0_2 \\ 
      \pa_t g_0 \\ 
      i \pa_t g_0 
    \end{pmatrix} 
    = \Pi_1 \begin{pmatrix} 
      0_2 \\ 
      -3\lambda (g_0^2 \bar{g}_1 + 2 |g_0|^2 g_1) \\ 
      0 
    \end{pmatrix}.
\end{split}
\end{align}
Noting that for any scalar functions $f$ and $g$, there holds
\begin{align} \label{ob1}
\begin{split}
    \Pi_1 A(\pa_x) \begin{pmatrix} 
      \nabla g \\ 
      0 \\ 
      0 
    \end{pmatrix} &= \Pi_1 \begin{pmatrix} 
      0_2 \\ 
      \Delta_x g \\ 
      0 
    \end{pmatrix}, \quad
    \Pi_1 A(\pa_x) 
    \begin{pmatrix} 
      0_2 \\ 
      f \\ 
      g 
    \end{pmatrix} = 0.
\end{split}
\end{align}
Then equation \eqref{P1U11} is equivalent to
\begin{align} \label{equationg1}
    2i \pa_t g_1 - \Delta_x g_1 + 3\la (g_0^2 \bar{g}_1 + 2 |g_0|^2 g_1) = 0.
\end{align}
 To ensure the initial difference between $U$ and $U_{a}$ is as small as possible, it is natural to impose $\Pi_1 U_{1,1}(0,\cdot) = 0$ in \eqref{U0_1},
which means
\begin{align*}
    g_1(0,\cdot) = 0.
\end{align*}
Then, the solution to \eqref{equationg1} is identically $0$:
\begin{align} \label{g1=0}%所有数字都要加上$$ \equiv section 不需要标号,他会自动标号 \section{} 即可
    g_1\equiv 0.
\end{align}
Then we have $f_1 = 0$. 
Applying $L_1^{-1}$ to \eqref{U11_eq} and using \eqref{f_1-P1U21}--\eqref{U2form} gives
\begin{align} \label{linear-shrodinger}
\begin{split}
    (I_4 - \Pi_1) U_{3,1} &= -\pa_t L_1^{-1} U_{1,1} + L_1^{-1} A(\pa_x) U_{2,1} 
    = \begin{pmatrix} 
      \nabla g_2 \\ 
      0 \\ 
      0 
    \end{pmatrix}.
\end{split}
\end{align}
For $\Pi_1 U_{3,1}$, there exists a scalar function $g_3$ such that
\begin{align*}
\Pi_1 U_{3,1} &=  g_3  e^+.
\end{align*}
Hence
\begin{align} \label{U31_form}
    U_{3,1} &= g_3 e^+ + \begin{pmatrix} 
\nabla g_2 \\ 
0 \\ 
0 
\end{pmatrix}.
\end{align}

For $p = 3$, equation \eqref{phi_1p} implies that
\begin{align} \label{U33_form}
\begin{split}
    U_{3,3} &= L_3^{-1} A(\pa_x) U_{2,3} 
    =  \nabla (g_0^3) \frac{\la}{8} \begin{pmatrix} 
    1_2 \\ 
    0 \\ 
    0 
    \end{pmatrix}.
\end{split}
\end{align}
Using \eqref{U31_form},  \eqref{U33_form} and Proposition \ref{unplarge},  $U_3$ takes the form
\begin{align} \label{U3form}
    U_3 &= e^{i\xth} \left[ g_3 e^+ + \begin{pmatrix} 
    \nabla g_2 \\ 
    0 \\ 
    0 
    \end{pmatrix} \right] + e^{3i\theta} \nabla (g_0^3) \frac{\la}{8} 
    \begin{pmatrix} 
    1_2 \\ 
    0 \\ 
    0 
    \end{pmatrix} 
     + \text{c.c.}.
\end{align} %为什么这个theta我写成th也不行 可以统一写成\xth 或者直接写theta看你  ,后面空一格

\subsection{$O(\epsilon^2)$}
The equation containing all terms of order $O(\e^2)$ is $\Phi_{2,p} = 0$:
\begin{align} \label{eq2p} %begin{align} \end{align} 也要同时修改
    \pa_t U_{2,p}-A(\pa_x) U_{3, p}+L_p U_{4,p}= F_{2,p},\quad \forall\,p \in \mathbb{Z},
\end{align}
 where   $f_2 := 3\la (u_0 u_1^2 + u_0^2 u_2)$.
By \eqref{U1form} and the fact  $u_1=g_{1} = 0$, \eqref{U0form} and \eqref{U2form}, we have %%不能超出页面范围
\begin{align*}
\begin{split}
    f_2 &= 3\la u_0^2 u_2 =3\la ( e^{i\xth} g_0 +e^{-i\xth} \bar{g}_0 )^2 ( e^{i\xth} g_2 + e^{-i\xth} \bar{g}_2+e^{3i\xth}\frac{\la}{8} g_0^3 + e^{-3i\xth}\frac{\la}{8} \bar{g}_0^3 ) \\
    &= 3\la \Big[ e^{5i\xth} \frac{\la}{8} g_0^5 + e^{3i\xth} ( g_0^2 g_2 + \frac{\la}{4} |g_0|^2 g_0^3 ) + e^{i\xth} ( g_0^2 \bar{g}_2 + 2 |g_0|^2 g_2 + \frac{\la}{8} |g_0|^4 g_0 ) \Big] + \text{c.c.} .
\end{split}
\end{align*}

For $p=1$, equation \eqref{eq2p} takes the form
\begin{align} \label{eq21}
    \pa_t U_{2,1} - A(\pa_x) U_{3,1} +L_1 U_{4,1} =
    \begin{pmatrix}
        0_2 \\
        -f_{2,1} \\
        0
    \end{pmatrix}.
\end{align}
In the right-hand side, $f_{2,1}$ is the leading harmonic of $f_2$:
\begin{align} \label{formf21}
    f_{2,1} := 3\la ( g_0^2 \bar{g}_2 + 2 |g_0|^2 g_2 + \frac{\la}{8} |g_0|^4 g_0 ).
\end{align}
Applying $\Pi_1$ to \eqref{eq21} and using the decomposition $U_{3,1} = \Pi_1 U_{3,1} +(I_4- \Pi_1) U_{3,1}$, together with \eqref{P1APx}, \eqref{ob1}, \eqref{linear-shrodinger}, there holds
\begin{align} \label{pi1U21}
    \pa_t \Pi_1 U_{2,1} -\Pi_1
    \begin{pmatrix}
        0_2 \\
        \Delta_x g_2 \\
        0
    \end{pmatrix}
    =\Pi_1
    \begin{pmatrix}
        0_2 \\
        -f_{2,1} \\
        0
    \end{pmatrix},
\end{align}
which is equivalent to (using \eqref{p1U21})
\begin{align} \label{ls2}
    2i \pa_t g_2 +\pa_{t t} g_0 - \Delta_x g_2 +f_{2,1} = 0.
\end{align}
We shall impose a suitable initial data for $g_2$ to ensure the difference between $U$ and $U_a$ is as small as possible. Thanks to the observation in \eqref{U0_1}, \eqref{g1=0}, we know that
\begin{align*}
    U ( 0,\cdot)-U_0(0,\cdot)-\e U_1(0,\cdot) = 0.
\end{align*}
To get a smallest difference, we choose $g_2$ such that
\begin{align} \label{U2=0}
    U_2(0, \cdot) = 0.
\end{align}
By \eqref{U2form} and $g_1(0, \cdot) = 0$, to achieve \eqref{U2=0}, it is sufficient to impose
\begin{align*}
    \begin{cases}
    &i g_2(0, \cdot) + \pa_t g_0(0, \cdot) + \frac{3i\la}{8} g_0^3(0, \cdot) + \text{c.c.} = 0, \\
    &g_2(0, \cdot) + \frac{\la}{8} g_0^3(0, \cdot) + \text{c.c.} = 0.
    \end{cases}
\end{align*}
This implies
\begin{align} \label{reim}
\begin{split}
\begin{cases}
    \Im(g_2(0, \cdot)) &= \Re \left( \pa_t g_0(0, \cdot) + \frac{3i\la}{8} g_0^3(0, \cdot) \right) = -\frac{1}{4} \Delta_x \psi + \frac{21\la}{64} \phi^2 \psi + \frac{9\la}{64} \psi^3, \\
    \Re(g_2(0, \cdot)) &= -\Re \left( \frac{\la}{8} g_0^3(0, \cdot) \right) = -\frac{\la}{64} ( \phi^3 - 3\phi \psi^2 ).
\end{cases}
\end{split}
\end{align}
By \eqref{U2form}, \eqref{U3form}, \eqref{pi1U21} and \eqref{ls2}, applying $L_1^{-1}$ to \eqref{eq21} gives%%, 后面空格 这意味着
\begin{align*}
\begin{split}
    (I_4 - \Pi_1) U_{4,1} = -\pa_t L_1^{-1} U_{2,1} + L_1^{-1} A(\pa_x) U_{3,1} + L_1^{-1} \begin{pmatrix} 
      0_2 \\ 
      -f_{2,1} \\ 
      0 
    \end{pmatrix} = 
    \begin{pmatrix} 
      \nabla g_3 \\ 
      \frac{\pa_t g_2}{2} \\ 
      \frac{i \pa_t g_2}{2} 
    \end{pmatrix}.
\end{split}
\end{align*}
For $\Pi_1 U_{4,1}$, there exists a scalar function $g_4$ such that
\begin{align*}
    \Pi_1 U_{4,1} &=( g_4 - \frac{i}{2} \pa_t g_2 ) e^+.
\end{align*}
Using the above two formulas, we can get the form of $U_{4,1}$.

For $p = 3$ and $p=5$, the calculation method is essentially the same as the previous ones. The detailed calculations can be referred to in \cite{BLZ24}. Here, for brevity, we omit the details and directly present the results.
\begin{comment}
\begin{align*} %%如果不引用就用align*
    L_3 U_{4,3} &= -\pa_t U_{2,3} + A(\pa_x) U_{3,3} - \begin{pmatrix} 
      0_2 \\ 
      f_{2,3} \\
      0 
    \end{pmatrix} \\
    &= - \pa_t (g_0^3) \frac{\la}{8} \begin{pmatrix} 
      0_2 \\ 
      3i \\ 
      1 
    \end{pmatrix} + 
    \frac{\la}{8} \Delta_x (g_0^3)
    \begin{pmatrix}
        0_2 \\
        1 \\
        0
    \end{pmatrix}
    -f_{2,3}
    \begin{pmatrix}
        0_2 \\
        1 \\
        0
    \end{pmatrix}. %%末尾的逗号不要忘记 这等价于
\end{align*}
In the ride side, $f_{2,3}$ is the order-three harmonic of $f_2$
\begin{align*}
    f_{2,3} := 3\la \left( g_0^2 g_2 + \frac{\la}{4} |g_0|^2 g_0^3 \right).
\end{align*}
Hence%%用英语的逗号 统一 没有引用的就用align* 引用了就要加上label
\end{comment}

For $p=3$,
\begin{align} \label{U43form}
    U_{4,3} &= - \pa_t (g_0^3) \frac{\la}{8} 
    \begin{pmatrix}
        0_2 \\
        \frac{5}{4} \\
        -\frac{3i}{4}
    \end{pmatrix}
    +\left[ \frac{\la}{8} \Delta_x (g_0^3)-3 \la (g_0^2 g_2 + \frac{\la}{4} |g_0|^2 g_0^3) \right] %%括号尽量都加上左右 left right
    \begin{pmatrix}
        0_2 \\
        -\frac{3i}{8} \\
        -\frac{1}{8}
    \end{pmatrix}. %%末尾的逗号不要忘记
\end{align} %%逗号前面不需要空格

For $p=5$, 
\begin{align} \label{U45form} %没有引用的就用align* 引用了就
    U_{4,5} = - \frac{3 \la ^2}{8} g_0^5
    \begin{pmatrix}
        0_2 \\
        -\frac{5i}{24} \\
        -\frac{1}{24}    
    \end{pmatrix}.
\end{align} %%统一$$前后不加空格 %没有引用的就用align* 引用了就要加上\label{} 我看老师那个上面每个式子都有编号
 %%所有的字母都需要 doralla符号$U_4$ 包括数字都需要$5$  括号都需要加上 \left[ \right]  而且参考一下老师 \left[ \right]  
 
To sum up, the form of $U_4$ is
\begin{align*} 
    U_4 = e^{i\xth}\left[g_4 e^+ +
    \begin{pmatrix}
        \nabla g_3 \\
        \pa_t g_2 \\
        0
    \end{pmatrix}
    \right] + e^{3i\xth} U_{4,3} + e^{5i\xth} U_{4,5} +\text{c.c.},
\end{align*}
where $U_{4,3}$ and $U_{4,5}$ are given in \eqref{U43form} and \eqref{U45form}, $g_2$ is the solution to \eqref{ls2}--\eqref{reim}.

\subsection{$O(\epsilon^n)$} \label{epn}

Based on the previous WKB expansion, we can deduce the specific form of the WKB solution by induction.   

\subsubsection{$U_{n,p}$ with $p \in \Zz_{\rm even}$}

By conclusion in Proposition \ref{Ueven}, we impose
\begin{align} \label{unpeven}
    U_{n,p} = 0, \quad \forall \, n \in \{0,\,1,\,2,\,\cdots,\,K+2 \}, \quad \forall \, p \in \mathbb{Z}_{\rm even}.
\end{align} %% .句号是在里面的 不是在外面

\subsubsection{$U_{n, \pm 1}$}
For the leading harmonic with $p = \pm 1$, we have
\begin{proposition}\label{Un-1}
    For $n = 0,\,1,\, 2,\, \ldots,\, K + 2$, there exist scalar functions $g_n$ such that
\begin{align*}
\begin{split}
    &(1 - \Pi_1) U_{n,1} = 
    \begin{pmatrix} 
      \nabla g_{n-1} \\
      \frac{1}{2} \pa_t g_{n-2} \\ 
      \frac{i}{2} \pa_t g_{n-2} 
    \end{pmatrix}, \quad \Pi_1 U_{n,1} = ( g_n - \frac{i}{2} \pa_t g_{n-2} ) e^+, \quad  U_{n,1} = g_n e^+ + 
    \begin{pmatrix} 
      \nabla g_{n-1} \\ 
      \pa_t g_{n-2} \\ 
      0 
    \end{pmatrix}.
\end{split}
\end{align*} %%schrodinger方程用英文
We impose $g_{-2}=g_{-1}=0$. Moreover, for all $n \in \{0,\, 1,\, 2,\, \ldots,\, K\}$, $g_n$ satisfies Schr\"{o}dinger equation
\begin{align} \label{induceeq} %%里面空一下
    2i \pa_t g_n + \pa_{t t} g_{n-2} - \Delta_x g_n + f(u_a)_{n,1} = 0.
\end{align}
\end{proposition} 
The proof of this proposition can be found in Section 3 of \cite{BLZ24}. We omit the details for brevity.
\subsubsection{$U_{n,p}$ with $p$ large}

By Proposition \ref{unplarge}, we know that $U_{n,p} = 0$ for $p\geq p(n)+1$.

\subsubsection{The specific form of $U_n$}

In Proposition \ref{Un-1}, we have got the specific form of $U_{n,\pm1}$. Now we consider the case $p \geq 3$. According to \eqref{unpeven} and conclusion in Proposition \ref{unplarge}, it is enough to consider $p \in \{3,\, 5,\, \ldots,\, p(n)\}$. The following proposition gives the specific form of the WKB solution and the induction relations.
\begin{proposition}\label{unpall}
    For all $n \in \{0,\, 1,\, 2,\, \ldots\, K\} \cap \mathbb{Z}_{\rm odd}$, we impose $g_n(0,\cdot) = 0$. Then,
    \begin{enumerate}
        \item For all $n \in \{0,\, 1,\, 2,\, \ldots,\, K\} \cap \mathbb{Z}_{\rm odd}$, $g_n = 0$.
        \item For all $n \in \{0,\, 1,\, 2,\, \ldots,\, K + 2\} \cap \mathbb{Z}_{\rm even}$, and $p \in \{ 3,\, 5,\, \ldots,\, n+1 \}$, the form of $U_{n,p}$ is
        \begin{align*}
            U_{n,p} 
            = \begin{pmatrix} 
              0_2 \\ 
              \pa_t u_{n-2,p} + i p u_{n,p} \\ 
              u_{n,p} 
            \end{pmatrix},
        \end{align*}
        with
        \begin{align} \label{Uk+4pp}
            u_{n,p} &= \frac{1}{1 - p^2} \left( -f(u_a)_{n-2,p} - \pa_t^2 u_{n-4,p} + \Delta_x u_{n-2,p} - 2ip \pa_t u_{n-2,p} \right).
        \end{align}
        For all $n \in \{0,\, 1,\, 2,\, \ldots,\, K + 2\} \cap \mathbb{Z}_{\rm even}$, the form of $U_n$ is
        \begin{align*}
            U_n = e^{i\xth} \left( g_n e^+ + 
            \begin{pmatrix} 
              0_2 \\ 
              \pa_t g_{n-2} \\ 
              0 
            \end{pmatrix} \right) + 
            \sum_{ p = 3 }^{ n + 1 } e^{ip\xth} U_{n,p} + \text{c.c.}.
        \end{align*}
        \item  For all $n \in \{0,\, 1,\, 2,\, \ldots,\, K + 2\} \cap \mathbb{Z}_{\rm odd}$ and $p \in \{ 3,\, 5,\, \ldots,\, n \}$, the form of $U_{n,p}$ is
        \begin{align*}
            U_{n,p} =  \begin{pmatrix} 
              \nabla u_{n-1, p} \\ 
              0 \\ 
              0 
            \end{pmatrix}.
        \end{align*}
         For all $n \in \{0, 1, 2, \ldots, K + 2\} \cap \mathbb{Z}_{\rm odd}$, the form of $U_n$ is
         \begin{align*}
            U_n = e^{i\xth} \left( g_n e^+ + 
            \begin{pmatrix} 
              \nabla g_{n-1} \\ 
              0 \\ 
              0 
            \end{pmatrix} \right) + 
            \sum_{ p = 3 }^{ n + 1 } e^{i p\xth} U_{n,p} + \text{c.c.}.
        \end{align*}
    \end{enumerate}
\end{proposition} 
Proposition \ref{unpall}  can be proved by induction, as showed in \cite{BLZ24}.

Introduce the notation 
\begin{align*}
    \pa_x^n := (\pa_x^\alpha)_{\alpha \in \mathbb{N}^2, |\alpha| \leq n},\quad \forall\, n \in \mathbb{Z}_+.
\end{align*}
By induction,  we can deduce (see \cite{BLZ24})  from Proposition \ref{unpall} that:
\begin{corollary} \label{poly}
    For all $n \in \{ 0,\,1,\,2, \ldots,\, K + 2 \} \cap \mathbb{Z}_{\rm odd}$, we impose the initial data $g_n(0, \cdot) = 0$. Then  for all $n \in \{ 0,\,1,\,2, \ldots,\, K + 2 \} \cap \mathbb{Z}_{\rm even}$ and $p \geq 3$,
    \begin{align*}
        u_{n,p} &= P_n \left( \pa_x^{n-2} (g_0, \bar {g}_0), \pa_x^{n-4} ( g_2, \bar g_2) \ldots , g_{n-2}, \bar g_{n-2} \right).
    \end{align*}
    Thus 
    \begin{enumerate}
        \item[(i)] For all $n \in \{ 2, \ldots ,\, K + 2 \} \cap \mathbb{Z}_{\rm even}$ and $p \geq 3$,
        \begin{align*}
            U_{n,p} &= P_n \left( \pa_x^{n-2} (g_0, \bar {g}_0), \pa_x^{n-4} ( g_2, \bar g_2) \ldots , g_{n-2}, \bar g_{n-2} \right).
        \end{align*}
        \item[(ii)] For all $n \in \{ 2, \ldots ,\, K + 2 \} \cap \mathbb{Z}_{\rm odd}$ and $p \geq 3$,
        \begin{align*}
            U_{n,p} &= P_n \left( \pa_x^{n-2} (g_0, \bar {g}_0), \pa_x^{n-4} ( g_2, \bar g_2) \ldots , g_{n-3}, \bar g_{n-3} \right).
        \end{align*} %公式里面的逗号如果感觉连得太近了 \, 用这个隔开, 如果是括号里面两个的话,就不用了,多个最好隔开一下
    \end{enumerate}
\end{corollary}
Here, we use $P_n$ to denote a polynomial whose degree depends on $n$.

\subsubsection{Initial data of $U_{n,p}$}

In \eqref{g0inital}, we have imposed the initial data $g_0 ( 0, \cdot) = \frac{\phi - i\psi}{2}$, $g_1=0$, then we have
\begin{align*}
    ( U_0 + \e U_1 )(0, \cdot) = (\e \nabla^{\rm T} \phi, \psi, \phi) = U(0, \cdot),
\end{align*} 
with $U(0, \cdot)$ defined in \eqref{inital_U}.
To ensure that the difference $U(0, \cdot) - U_a(0, \cdot)$ is small, we have chosen $g_2(0, \cdot)$ in \eqref{reim} to ensure $U_2(0, \cdot) = 0$. Moreover, we set $g_3(0, \cdot) = 0$ and deduce from \eqref{U3form} that
\begin{align*}
    U_3(0, \cdot) = 0.
\end{align*}
By induction, we can show that by choosing initial data of $g_n$ properly, we can ensure that for $n = 2,\, 3,\, \cdots, K$, $U_n(0, \cdot) = 0$. 
\begin{proposition}\label{initalUn}
    Let $g_n(0, \cdot) = 0$ for all $n \in \{1,\, 2,\, \cdots,\, K\} \cap \mathbb{Z}_{\rm odd}$. Then, for all $n \in \{2,\, \cdots,\, K\} \cap \mathbb{Z}_{\rm even}$, there exists a polynomial $P_n(\pa^n_x \phi, \pa^n_x \psi)$, by choosing initial data $g_n(0, \cdot) = P_n(\pa^n_x \phi, \pa^n_x \psi)$, we can ensure
    \begin{align} \label{Uninital}
        U_n(0, \cdot) = 0, \quad \forall\,  n \in \{2,\, \cdots,\, K,\, K + 1\}.
    \end{align}
\end{proposition} 
\begin{proof}
    We know from Proposition \ref{unpall} that for odd $n$, $g_n = 0$. We have shown \eqref{Uninital} for $n = 1,\, 2,\, 3$. In particular, $P_1 = g_0(0, \cdot)$ and $P_3 = g_2(0, \cdot)$ have been given in \eqref{g0inital} and \eqref{reim}. Suppose $k$ is odd  and satisfies $2 \leq k \leq K$. We assume \eqref{Uninital} holds for $n \leq k$. Now we show there exists a polynomial $g_{k+1}(0,\cdot)=P(\pa^k_x \phi, \pa^k_x \psi)$ such that 
    \begin{align*}
        U_{k+1}(0, \cdot) = U_{k+2}(0, \cdot) = 0.
    \end{align*}
    By the second result in Proposition \ref{Un-1}, we know that the form of $U_{k+1}$ is
    \begin{align*}
        U_{k+1} = e^{i\xth} \left[ g_{k+1} e^+ + 
        \begin{pmatrix} 
        0_2 \\ 
        \pa_t g_{k-1} \\ 
        0 
        \end{pmatrix} \right] + \sum_{p=3}^{k+2} e^{ip\xth} \begin{pmatrix} 
        0_2 \\ 
        \pa_t u_{k-1,p} + ip u_{k+1,p} \\ 
        u_{k+1,p} 
        \end{pmatrix} + \text{c.c.}.
    \end{align*}
    To ensure $U_{k+1}(0, \cdot) = 0$, it is sufficient to impose $u_{k+1} = v_{k+1} = 0$. Specifically, 
    \begin{align*}
        \begin{cases}
          ig_{k+1}(0, \cdot) + \pa_t g_{k-1}(0, \cdot) + \sum_{p=3}^{k+2} (\pa_t u_{k-1,p} + ip u_{k+1,p})(0, \cdot) + \text{c.c.} = 0, \\
          g_{k+1}(0, \cdot) + \sum_{p=3}^{k+2} u_{k+1,p}(0, \cdot) + \text{c.c.} = 0.
        \end{cases}
    \end{align*}
    Separate the equation into its real and imaginary parts
    \begin{align} \label{gk+1form}
    \begin{split}
        \begin{cases} %%Re Im 用的是\Re \Im
          \Im (g_{k+1}(0, \cdot)) = \Re \left( \pa_t g_{k-1}(0, \cdot) + \sum_{p=3}^{k+2} (\pa_t u_{k-1,p} + ip u_{k+1,p})(0, \cdot) \right), \\
          \Re (g_{k+1}(0, \cdot)) = -\Re \left( \sum_{p=3}^{k+2} u_{k+1,p}(0, \cdot) \right).
        \end{cases}
    \end{split}
    \end{align}
    By Proposition \ref{Un-1}, Corollary \ref{poly} and induction assumption, together with the Schr\"{o}dinger equation \eqref{induceeq} with $n \leq k - 1$,  $g_{k+1}(0, \cdot)$ can be written as polynomials in $(\pa^k_x \phi, \pa^k_x \psi)$. 
    Next, we show that the initial data given in \eqref{gk+1form} can ensure $U_{k+2}(0, \cdot) = 0$. By Proposition \ref{Un-1}, we can deduce
    \begin{align*}
        U_{k+2} = e^{i\xth} 
        \begin{pmatrix} 
        \nabla g_{k+1} \\ 
        0 \\ 
        0 
        \end{pmatrix} + 
        \sum_{p=3}^{k+2} e^{ip\xth} 
        \begin{pmatrix} 
        \nabla u_{k+1,p} \\ 
        0 \\ 
        0 
        \end{pmatrix} + \text{c.c.}.
    \end{align*}
    Together with the second equation in \eqref{gk+1form}, we have $U_{k+2}(0, \cdot) = 0$. The proof is completed.
\end{proof}

\subsection{Regularity of WKB solution: focusing case}
Now we summarize the results in Section \ref{epn} and give the regularity of the WKB solution. As shown in \eqref{poly}, each $U_{n,p}$ can be written as a polynomial of $g_0,\, g_2,\, \cdots,\, g_n$ and their derivatives up to order $n-2$. Hence, the regularity of $U_{n,p}$ is determined by $g_0,\, g_2,\, \cdots,\, g_n$. We focus on the focusing case ($\la < 0$).
\begin{proposition}\label{regf}
    Let $\la < 0$, $K \in \mathbb{Z}_{+} \cap \mathbb{Z}_{\rm even}$. For all $n \in \{0,\, 1,\, 2,\, \cdots,\, K\} \cap \mathbb{Z}_{\rm odd}$,  let $g_n \equiv 0$.  For all $n \in \{0,\, 1,\, 2,\, \cdots,\, K\} \cap \mathbb{Z}_{\rm even}$, let $g_n$ be the solution to the Schr\"{o}dinger equation \eqref{induceeq} with initial data $g_n(0, \cdot)$ be chosen as in Proposition \ref{initalUn}. Then
    \begin{align} \label{est_gn_f}
        g_n \in C([0, T^*); H^{s - 2n}(\mathbb{R}^2)), \quad \forall\,  n \in \{2, \cdots, K\} \cap \mathbb{Z}_{\rm even}.
    \end{align}
    Moreover, there exists a WKB solution $U_a$ taking the form in \eqref{WKB_inital}, with $U_{n,p}$ satisfying all the properties in Propositions \ref{unplarge}, \ref{Un-1}, \ref{unpall}, \ref{initalUn} and Corollary \ref{poly}. In particular, for all $n = 0,\, 1,\, 2,\, \cdots,\, K + 2 $ and $p \in H_n$, the amplitudes $U_{n,p}$ satisfy %%这里也
    \begin{align} \label{regunpf}
        U_{n,p} \in C([0, T^*); H^{s - 2n}(\mathbb{R}^2)).
    \end{align}
    The initial difference satisfies
    \begin{align*}
        \|U(0, \cdot) - U_a(0, \cdot)\|_{H^{s - 2K - 2}} \leq C_{\phi, \psi} \e^{K + 2}. %%为什么这里突然要用varepsilon了啊,需要统一
    \end{align*}
\end{proposition}  %%这个句号不能出现在这里
\begin{proof}
    By Proposition \ref{local_wp}, the cubic Schr\"{o}dinger equation \eqref{ls_g_0} admits a unique solution
    \begin{align*}
        g_0 \in C([0, T^*), H^{s}(\mathbb{R}^2)),
    \end{align*}

    By Proposition \ref{initalUn} and the fact that for $s > 1$, $H^s(\mathbb{R}^2)$ is a Banach algebra, we have
    \begin{align} \label{est_initial}
        \|g_n(0, \cdot)\|_{H^{s - n}} \leq C_{\phi, \psi}, \quad \forall\,  n \in \{2,\, \cdots,\, K\} \cap \Zz_{\rm even}.
    \end{align}
    
    In particular, the function $g_2$ satisfies the linear Schr\"{o}dinger equation \eqref{ls2} with initial data \eqref{reim}. Using Duhamel's formula, we can write the specific form of $g_2$ as%% Duhamel这个不用$$
    \begin{align} \label{duhamelg2}
        g_2(t, \cdot) = 
        e^{-\frac{ it \Delta_x}{2}} g_2(0, \cdot) + 
        \int_0^t e^{-\frac{ i(t-s) \Delta_x}{2}} \frac{i}{2} ( \pa_{t t} g_0 + f_{2,1} )(s, \cdot)  \D{s}.
    \end{align} %积分号的 dt 需要写成 \D{s}
  By applying $\d_{t}$ to the Schr\"{o}dinger equation in $g_0$, we have $\pa_{t t} g_0 \in C([0, T^*), H^{s - 4}(\mathbb{R}^2))$. According to the form of $f_{2,1}$ in \eqref{formf21}, we deduce that
    \begin{align*}
        \|f_{2,1}(s, \cdot)\|_{H^{s - 4}} &\leq 3\la \big( \|g_0^2 \bar{g}_2\|_{H^{s - 4}}  + 2 \| |g_0|^2 g_2\|_{H^{s - 4}}  + \frac{\la}{8} \||g_0|^4 g_0\|_{H^{s - 4}}  \big) \\
        &\leq 3\la \big( (\|g_0^2 \|_{H^{s - 4}}  + 2 \| |g_0|^2 \|_{H^{s - 4}}  )\|g_2\|_{H^{s - 4}}  + \frac{\la}{8} \||g_0|^4 g_0\|_{H^{s - 4}}  \big).
    \end{align*}
 As a result, 
    \begin{align} \label{estg2h}
    \begin{split}
        \|g_2( t, \cdot)\|_{H^{s - 4}} 
        &\leq \|g_2(0, \cdot)\|_{H^{s - 4}} + 
        \int_0^t \frac{1}{2}( \|\pa_{t t} g_0(s, \cdot)\|_{H^{s - 4}} + 
        \|f_{2,1}(s, \cdot)\| ) \D{s} \\
        &\leq C_{\phi, \psi} + 
        \int_0^t C_{\phi, \psi}(s) ( 1 + \|g_2(s, \cdot)\|_{H^{s - 4}} ) \D{s}.
    \end{split}
    \end{align} 
    Using Gronwall's inequality in \eqref{estg2h} gives %%gronwall不等式不需要 $$符号
    \begin{align} \label{estg2l}
        \|g_2\|_{L^\infty((0, t); H^{s- 4})} \leq C_{\phi, \psi}(t), \quad \forall \, t \in (0, T^*).
    \end{align}
    Hence, $g_2 \in C([0, T^*); H^{s - 4}(\mathbb{R}^2))$. The continuity in time follows from \eqref{duhamelg2}.

    Up to this point, we have proved \eqref{est_gn_f} with $n=2$. Then we consider the case with $n \ge 2$.
    
    By the form of $f$ in \eqref{form_ff}, we know that
    \begin{align}
    \begin{split} \label{decomposition_f}
        f(u_a)_{n,1} 
        = \la (3u_{0, 1}^2 u_{n, -1} + 6u_{0, 1} u_{0, -1} u_{n, 1}) + f(u_a)_{n,1}^{(r)},
    \end{split}
    \end{align}
    where
    \begin{align*}
        f(u_a)_{n,1}^{(r)} := \la \sum_{n_1 + n_2 + n_3 = n; n_1, n_2, n_3 \leq n-1} \sum_{p_1 + p_2 + p_3 = 1} u_{n_1, p_1} u_{n_2, p_2} u_{n_3, p_3}.
    \end{align*}
    By Proposition \ref{Un-1}, we deduce that  for all $n = 0,\, 1,\, 2,\, \cdots,\, K + 2$, there holds $u_{n, 1} = g_n$. 
    Thus
    \begin{align*}
        f(u_a)_{n,1} = \la (3g_0^2 \bar{g}_n + 6|g_0|^2 g_n) +\la \sum_{n_1 + n_2 + n_3 = n; n_1, n_2, n_3 \leq n-1} \sum_{p_1 + p_2 + p_3 = 1} u_{n_1, p_1} u_{n_2, p_2} u_{n_3, p_3}.
    \end{align*} %%这里最后两行，求和号需要改一下下标
    By Corollary \ref{poly}, we know that for all $k \leq n-1$ and $p \in H_k \cap \Zz_{\geq 3}$, $u_{k, p}$ is a polynomial in $g_0,\, g_2,\, \cdots,\, g_{k-2}$. Thus, similar to  \eqref{duhamelg2}--\eqref{estg2l}, we can use induction argument to prove that for all $n \in \{2,\, \cdots,\, K\}\cap \Zz_{\rm even}$, there holds $g_n \in C([0, T^*); H^{s - 2n}(\mathbb{R}^2))$.
    
    Imposing that $g_{K+1} = g_{K+2} = 0$ and using Propositions \ref{unplarge}--\ref{regf}, we can construct a WKB solution satisfying \eqref{regunpf}.

   Moreover, by Proposition \ref{initalUn}, the initial perturbation satisfies
    \begin{align} \label{inital per}
        U_a(0, \cdot) - U(0, \cdot) = \e^{K+2} U_{K+2}(0, \cdot).
    \end{align}
    The proof is completed. 
\end{proof}

\subsection{Regularity of WKB solutions: defocusing case}

For the defocusing case ($\la > 0$), we can construct a global WKB solution. 
\begin{proposition}\label{regd}
    Let $\la > 0$, $K \in \mathbb{Z}_{+} \cap \mathbb{Z}_{\rm even}$. Suppose $(\phi, \psi)$ satisfies \eqref{inital_datum}-\eqref{weighted} with $s > 2K + 5$. For all $n \in \{0,\, 1,\, 2,\, \cdots,\, K\} \cap \mathbb{Z}_{\rm odd}$,  let $g_n \equiv 0$.  For all $n \in \{0,\, 1,\, 2,\, \cdots,\, K\} \cap \mathbb{Z}_{\rm even}$, let $g_n$ be the solution to the Schr\"{o}dinger equation \eqref{induceeq} with initial data $g_n(0, \cdot)$ be chosen as in Proposition \ref{initalUn}. Then $g_0$ satisfies the uniform estimates and decay estimates stated in Proposition \ref{global_wp}, and 
        \begin{align*}
        g_n \in C([0, \infty); H^{s - 2n}(\mathbb{R}^2)), \quad \forall n \in \{2,\, \cdots,\, K\} \cap \mathbb{Z}_{\rm even}.
    \end{align*}%% 逗号后面别忘记空格

   Let $u_{n,p}$ be the approximate solutions defined in \eqref{WKB_inital}. Then for all $\,t\in (0,\infty)$,  there holds
 for all $n \in \{0,\, 2,\, 4,\, \cdots,\, K\}$ that
  \begin{align} \label{estd1}
        \|\pa_t^j g_n(t, \cdot)\|_{H^{s - 2n - 2j}(\mathbb{R}^2)} \leq C_{\phi, \psi} \left( 1 + (1 + t)^{\frac{n}{2} \hat{N}_{\phi,\psi} + \frac{n-2}{2}} \right), \quad \forall j \leq \frac{s}{2} - n,
    \end{align}
  and  for all  $n \in \{2,\, 4,\, \cdots,\, K + 2\}$ that
    \begin{align} \label{estd2}
    \begin{split}
         \|\pa_t^j u_{n,p}(t, \cdot)\|_{H^{s - 2n + 4 - 2j}(\mathbb{R}^2)} \leq C_{\phi, \psi} \left( 1 + (1 + t)^{\frac{n-2}{2} \hat{N}_{\phi,\psi} + \frac{n-4}{2}} \right),  \quad  \forall j \leq \frac{s}{2} - n + 2,  \,  p \geq 3.
    \end{split}
    \end{align}
\end{proposition} 
\begin{proof}
   By Proposition \ref{global_wp}, we know that $g_0 \in L^\infty((0, \infty); H^{s}(\mathbb{R}^2))$ and as $t \to \infty$, $\|g_0(t)\|_{L^\infty(\mathbb{R}^2)}$ decays with rate $t^{-1}$. The continuity of $g_n$ in time follows from equation \eqref{induceeq} naturally. Then it is sufficient to prove estimates \eqref{estd1} and \eqref{estd2}. We will use induction to prove our desired estimates.

Using equation \eqref{ls_g_0} gives
    \begin{align*}
        \|\pa_t g_0 (0, \cdot)\|_{H^{s-2}} \leq C \|\Delta_x g_0 (t, \cdot)\|_{H^{s-2}} + C \| |g_0(t, \cdot)|^2 g_0(t, \cdot) \|_{H^{s-2}} \leq C_{\phi, \psi}.
    \end{align*}
    Here and in the sequel, we use repeatedly the classical estimates
       \begin{align} \label{classical}
     \begin{split}
         &\|u v\|_{H^s} \leq M(s) ( \|u\|_{H^s} \|v\|_{L^\infty} + \|u\|_{L^\infty} \|v\|_{H^s} ), \quad \forall \, s >  0, \\
        &\|u\|_{L^\infty} \leq M(s-1 ) \|u\|_{H^s}, \quad \forall \, s > 1,
     \end{split}   
    \end{align}
   where the constant $M(s)$  is decreasing in $(0, \infty)$ and $M(s) \to +\infty$ as $s \to 0^+$.

    Applying $\pa_t^j$ to equation \eqref{ls_g_0} and using continuation argument, we have
    \begin{align} \label{estg0}
        \|\pa_t^j g_0(t, \cdot)\|_{H^{s - 2j}} \leq C_{\phi, \psi}, \quad \forall \, j \leq \frac{s}{2}.
    \end{align}
    This is exactly \eqref{estd1} with $n = 0$.
    
    Then we consider the estimate of $g_2$. Using Duhamel's formula in equations \eqref{ls2} and \eqref{reim}, we arrive at an estimate as \eqref{estg2h}
    \begin{align} \label{estg2hd}
    \begin{split}
        \|g_2(t, \cdot)\|_{H^{s - 4}} \leq &\|g_2(0, \cdot)\|_{H^{s - 4}} 
        +\int_0^t \frac{1}{2} ( \|\pa_{t t} g_0(s, \cdot)\|_{H^{s - 4}} 
        + \|f_{2,1}(s, \cdot)\|_{H^{s - 4}} )  \D{s} \\
        \leq &C_{\phi, \psi}(1+t) + \int_0^t \frac{1}{2}  \|f_{2,1}(s, \cdot)\|_{H^{s - 4}}  \, \D{s}, \quad \forall\, t \in (0, \infty).
    \end{split}
    \end{align}
     We have already got the specific form of  $f_{2,1}$ in \eqref{formf21}. Using Proposition \ref{global_wp} along with \eqref{classical}, it follows that
     \begin{align} \label{estimate_def21}
     \begin{split}
         \|f_{2,1}(t, \cdot)\|_{H^{s - 4}} 
         \leq &3\la \big( \|g_0^2 \bar{g}_2\|_{H^{s - 4}} + 2\||g_0|^2 g_2\|_{H^{s - 4}} + \frac{\la}{8} \||g_0|^4 g_0\|_{H^{s - 4}} \big) \\
         \leq & 9\la \tilde{M}(s - 5) \|g_0\|_{H^{s - 4}} \|g_0\|_{L^\infty} \|g_2\|_{H^{s - 4}} + C\|g_0\|_{H^{s-4}} \|g_0\|^4_{L^\infty} \\
         \le & 2N_1 (1 + t)^{-1} \|g_2(t, \cdot)\|_{H^{s - 4}}  + C_{\phi, \psi} (1 + t)^{-4},
    \end{split}
     \end{align}
     where 
     \begin{align} \label{def_N1}
         N_1=\frac{1}{2}\max \{2, 9\la \tilde{M}(s - 5) C_1(\phi, \psi, s - 4) C_2(\phi, \psi)\}.
     \end{align}
      Here $C_1(\phi, \psi, s - 4)$, $C_2(\phi, \psi)$ are defined in Proposition  \ref{global_wp} and $\tilde{M}(s)$ enjoys the same property as $M(s)$: $\tilde M(s) \to +\infty$ as $s \to 0^+$.
    
     \begin{remark}
         We impose a lower bound  $N_1 \geq 1$ to simplify the computation below. In fact, this lower bound holds in most cases, unless the initial data satisfy certain  smallness condition.
     \end{remark}
     From \eqref{estg2hd} and \eqref{estimate_def21}, we can write the estimate of $g_2$ as
     \begin{align*} 
     \begin{split}
         \|g_2(t, \cdot)\|_{H^{s - 4}} 
         &\leq C_{\phi, \psi} (1 + t) + \int_0^t N_1 (1 + t')^{-1} \|g_2(t', \cdot)\| _{H^{s-4}} + C_{\phi, \psi} (1 + t')^{-4} \D{t'} \\
         &\leq C_{\phi, \psi} (1 + t) + \int_0^t N_1 (1 + t')^{-1} \|g_2(t', \cdot)\|_{H^{s - 4}} \, \D{t'}.
     \end{split}
     \end{align*}
     Using Gronwall's inequality gives
     \begin{align} \label{estg2fin}
     \begin{split}
                  \|g_2(t, \cdot)\|_{H^{s - 4}} 
         &\leq C_{\phi, \psi} (1 + t) + \int_0^t C_{\phi, \psi} \left( \frac{ 1 + t } { 1 + t' }\right)^{N_1} \D{t'}\\
       &  \leq C_{\phi, \psi} (1 + t)^{N_1},
     \end{split}
     \end{align}
 which is \eqref{estd1} with $n = 2$, $j = 0$, with $\hat{N}_{\phi,\psi} \geq  N_1$.
    
     We use the Schr\"{o}dinger equation \eqref{ls2} to derive
     \begin{align} \label{g_2_s-6}
     \begin{split}
         \|\pa_t g_2(t, \cdot)\|_{H^{s - 6}} 
         &\leq C ( \|\pa_{t t} g_0(t, \cdot)\|_{H^{s - 6}} + \|\Delta_x g_2(t, \cdot)\|_{H^{s - 6}} + \|f_{2,1}(t, \cdot)\|_{H^{s - 6}} ) \\
         &\leq C_{\phi, \psi}  (1 + t)^{\hat{N}_{\phi,\psi}}.
     \end{split}
     \end{align}
     Applying $\pa_t^j$ to \eqref{ls2} and using induction argument gives
     \begin{align*}
          \|\pa_t^j g_2(t, \cdot)\|_{H^{s - 4 - 2j}} \leq C_{\phi, \psi} (1 + t)^{\hat{N}_{\phi,\psi}}, \quad \forall \, j \leq \frac{s}{2} - 2.
     \end{align*}
     This is \eqref{estd1} with $n = 2$.
     By \eqref{U2form}, \eqref{classical} and Proposition \ref{global_wp}, there holds
     \begin{align*}
         \|u_{2,3}(t, \cdot)\|_{H^{s}} 
         \leq \frac{\la}{8} \left \|g_{0}^{3}(t, \cdot) \right \|_{H^{s}} 
         \leq C \left \|g_{0}(t, \cdot) \right \|_{L^{\infty}}^{2} \left \|g_{0}(t, \cdot) \right \|_{H^{s}} 
         \leq C_{\phi, \psi}(1+t)^{-2}.
     \end{align*}
     By the estimate of $\pa_t^j g_0(t,\cdot)$ in \eqref{estg0}, we can get
     \begin{align*}
 \| \pa_t^j u_{2,3}(t, \cdot)\|_{H^{s-2j}} 
         \leq  C_{\phi,\psi}.
     \end{align*}
      Similarly, by \eqref{U43form} and \eqref{U45form}, these findings are presented
     \begin{align*}
         \left \| \pa _ {t}^{j} u_{4,3}(t, \cdot) \right \|_{H^{s-4-2j}} 
         &\leq C_{\phi, \psi}(1 + t)^{\hat{N}_{\phi, \psi}}, 
          \quad \forall j \leq \frac{s}{2}-2, \\ 
         \|\pa_{t}^{j} u_{4,5}(t, \cdot)\|_{H^{s-4-2j}} 
         &\leq C_{\phi, \psi}, \quad \forall\, j \leq \frac{s}{2}.    
     \end{align*}
     Now, we have proved \eqref{estd1} with $n =0,\, 2$  and \eqref{estd2} with $n = 0,\, 2,\, 4$. Let $k \geq 2$ be an even integer. Assume that \eqref{estd1} holds for $n \leq k$ and \eqref{estd2} holds for $n \leq k + 2$. Then we will prove  \eqref{estd1} for $n = k + 2$ and \eqref{estd2} for $n = k + 4$. 
     
  Recall the equation of $g_{k + 2}$ in \eqref{induceeq}. The nonlinearity $f(u_a)_{k+2,1}$ takes the form in \eqref{decomposition_f}, with $f (u _ {a} ) _ { k + 2 , 1 } ^ {(0)}:=3 \la ( g _ {0} ^ {2} \bar{g} _ {k+2} + 2 \left| g _ {0} \right| ^ {2} g_{k+2})$.
 Here we adopt the conclusion in Proposition \ref{unpall} that $g_n = 0$ for all odd integers $n$. 

     Similar to \eqref{estimate_def21}, we can get that for all $n \in \{2,\, 4,\, \cdots,\, K\}$, 
     \begin{align*}
     \begin{split}
         &\|f(u_a)_{n,1}^{(0)}(t, \cdot)\|_{H^{s - 2n}} \leq 3\la ( \|g_0^2 \bar{g}_n\|_{H^{s - 2n}} + 2\||g_0|^2 g_n\|_{H^{s - 2n}} ) \\
         &\le 9\la \tilde{M}(s - 2n - 1) \|g_0\|_{H^{s - 2n}} \|g_0\|_{L^\infty} \|g_n\|_{H^{s - 2n}} \\
         &\le 9\la \tilde{M}(s - 2n - 1) C_1(\phi, \psi, s - 2n) C_2(\phi, \psi) (1 + t)^{-1} \|g_n(t, \cdot)\|_{H^{s - 2n}},
     \end{split}
     \end{align*} %%逗号不要写在这里
     where $\tilde{M}(s - 2n - 1)$, $C_1(\phi, \psi, s - 2n)$, $C_2(\phi, \psi)$ are defined as in \eqref{estimate_def21}. Since $s > 2K + 5$, we can deduce
     \begin{align*}
         \tilde M(s - 2n - 1) \leq \tilde M(s - 2K - 1) \leq \tilde M(4) < \infty, \quad \forall \, n = 2,\, 4,\, \cdots,\, K.
     \end{align*}
     Thus
     \begin{align} \label{estf21hd}
         \|f(u_a)_{n,1}^{(0)}(t, \cdot)\|_{H^{s - 2n}} \leq 2 \hat{N}_{\phi,\psi} (1 + t)^{-1} \|g_n(t, \cdot)\|_{H^{s - 2n}}, \quad \forall \, n = 2,\, 4,\, \cdots,\, K,
     \end{align}
   with $\hat{N}_{\phi,\psi}$ defined as
     \begin{align} \label{definehatN}
         \hat{N}_{\phi,\psi} = \frac{1}{2} \max \big\{2, 9\lambda \tilde{M}(s - 2n - 1) C_1(\phi, \psi, s - 2n) C_2(\phi, \psi)| n = 2\cdots, K \big\}.
     \end{align}
     Then $\hat{N}_{\phi,\psi} \ge N_1$ holds naturally, where $N_1$ is defined in \eqref{def_N1}.
     
     Note that in $ f\left ( u _ {a} \right) _ { k + 2 , 1} ^ {(r)}$,  at most one $n_i$ equals $0$. Thus we have the decomposition below
     \begin{align*} 
          f (u _ {a})_{k + 2 , 1}^{(r)}:= & 3 \la \sum_{n_{1}+n_{2}=k+2 ; 2 \leq n_{1}, n_{2} \leq k} \sum_{p_{1}+p_{2}+p_{3}=1} u_{n_{1}, p_{1}} u_{n_{2}, p_{2}} u_{0, p_{3}} \\ 
          & +\la \sum_{n_{1}+n_{2}+n_{3}=k+2; 2 \leq n_{1}, n_{2}, n_{3} \leq k} \sum_{p_{1}+p_{2}+p_{3}=1} u_{n_{1}, p_{1}} u_{n_{2}, p_{2}} u_{n_{3}, p_{3}}.
    \end{align*}
    By induction assumption, for all $(n_1,\,n_2)$ with $n_1 + n_2 = k+2, 2 \leq n_1,\, n_2 \leq k$, there holds
    \begin{align} \label{remainder_1}
    \begin{split}
        &\|u_{n_{1}, p_{1}}(t, \cdot) u_{n_{2}, p_{2}}(t, \cdot) u_{0, p_{3}}(t, \cdot)\|_{H^{s-2 k}} \\ 
         &\leq C\|u_{n_{1}, p_{1}}(t, \cdot)\| _ {H ^ {s-2 k}}\|u_{n_{2}, p_{2}}(t, \cdot)\|_{H^{s-2 k}}\|u_{0, p_{3}}(t, \cdot)\|_{H^{s-2 k}} \\ 
        &\leq C_{\phi, \psi}(1+t)^{\frac{n_{1}+n_{2}}{2} \hat{N}_{\phi, \psi}+\frac{n_{1}+n_{2}-4}{2}} \\ 
        &\le C_{\phi, \psi}(1+t)^{\frac{k+2}{2} \hat{N}_{\phi, \psi}+\frac{k-2}{2}}.
    \end{split}
    \end{align}
    While for all $(n_1,\,n_2,\,n_3)$ with $n_1 + n_2 + n_3 = k+2, 2 \leq n_1,\, n_2,\, n_3 \leq k$, there holds
    \begin{align} \label{remainder_2}
    \begin{split}
           &\|u_{n_{1}, p_{1}}(t, \cdot) u_{n_{2}, p_{2}}(t, \cdot) u_{n_{3}, p_{3}}(t, \cdot)\|_{H^{s-2 k}} \\
        &\leq C\|u_{n_{1}, p_{1}}(t, \cdot)\|_{H^{s-2 k}}\|u_{n_{2}, p_{2}}(t, \cdot)\|_{H^{s-2 k}}\|u_{n_{3}, p_{3}}(t, \cdot)\|_{H^{s-2 k}} \\ 
        &\leq C_{\phi, \psi}(1+t)^{\frac{n_{1}+n_{2}+n_{3}}{2} \hat{N}_{\phi, \psi}+\frac{n_{1}+n_{2}+n_{3}-6}{2}} \\ 
        &\leq C_{\phi, \psi}(1+t)^{\frac{k+2}{2} \hat{N}_{\phi, \psi}+\frac{k-4}{2}}.
    \end{split}
    \end{align}
    Using the above two estimates \eqref{remainder_1} --\eqref{remainder_2} gives
    \begin{align} \label{est_fuak+21}
    \|f(u_{a})_{k + 2,1}^{(r)} \| _ {H^{s-2 k}} \leq C_{\phi, \psi}(1+t)^{\frac{k+2}{2} \hat{N}_{\phi, \psi} + \frac{k-2}{2}}.
    \end{align}
    Using the estimate of $f (u _ {a} ) _ { k + 2 , 1 } $ obtained from \eqref{estf21hd}, \eqref{decomposition_f}, \eqref{est_fuak+21} and following from \eqref{estg2fin}, we can get \eqref{estd1} with $n = k + 2,\, j = 0$, which is
    \begin{align} \label{gk+2est1}
        \| g _ { k + 2} (t, \cdot) \| _ {H^{s - 2(k+2)}}  
        \leq C _ {\phi, \psi}( 1 + t )^{\frac{k+2} {2} \hat{N}_{\phi, \psi}+\frac{k}{2}}.
    \end{align}
    In addition, the estimate of time derivatives follows from \eqref{induceeq} and is of  the form  \eqref{g_2_s-6}.
    By induction assumption, we have
    \begin{align*}
        \| \pa_{tt} g _ {k} (t, \cdot) \| _ {H^{ s-2k-6}} 
        \leq C_{\phi, \psi} (1 + t) ^{\frac{k}{2} \hat{N}_{\phi, \psi} + \frac{k-2}{2}} .
    \end{align*}
    From \eqref{gk+2est1}, we can get
    \begin{align*}
        \| \Delta g _ {k + 2}(t, \cdot) \| _ {H^{s - 2k - 6}} 
        &\leq \|g_{k + 2}(t, \cdot) \| _ {H^{s- 2k -4}} \leq C_{\phi, \psi}  (1 + t)^{\frac{k+2} {2} \hat{N}_{\phi, \psi}+\frac{k}{2}}.
    \end{align*}
    Using \eqref{decomposition_f}, \eqref{est_fuak+21} and combining with \eqref{gk+2est1}, we arrive at
    \begin{align*}
        \| f ( u _ {a} ) _ {k + 2 , 1}(t, \cdot) \| _ {H^{s - 2k - 6}} 
        \leq C _ {\phi, \psi} (1 + t) ^ {\frac{k + 2}{2} \hat{N} _ {\phi, \psi} + \frac{k} {2}} .
    \end{align*}
    To summarize,
    \begin{align*}
        \| \pa_{t} g_{k + 2}(t, \cdot) \| _ {H^{s - 2(k +2 ) - 2}} \leq C_{\phi, \psi} (1 + t)^{\frac{k + 2}{2} \hat{N}_{\phi, \psi}+\frac{k}{2}}.
    \end{align*}
    We can apply $\pa_t^j$ to \eqref{induceeq}  and use induction argument to derive
    \begin{align*}
        \|\pa_t^j g_n(t, \cdot)\|_{H^{s - 2k - 4 - 2j}(\mathbb{R}^2)} \leq C_{\phi, \psi} \left( 1 + (1 + t)^{\frac{k+2}{2} \hat{N}_{\phi,\psi} + \frac{k}{2}} \right).
    \end{align*}
    Then we show \eqref{estd2} with $n = k + 4$.    Again by induction assumption, we can get estimates of terms in the right-hand side of \eqref{Uk+4pp}:
    \begin{align*}
        &\|\pa_{t t} u_{k, p}\| _ {H^{s - 2(k + 4) + 4}} 
         \leq C _ {\phi, \psi} \left(1 +  (1 + t) ^ {\frac{k - 2}{2} \hat{N}_{\phi, \psi} + \frac{(k - 4)}{2}}\right), \\ 
        &\| \Delta_{x} u_{k + 2, p}\| _ {H^{s-2(k + 4) + 4}} 
         \leq C_{\phi, \psi} \left(1 +  (1 + t)^{\frac{k}{2} \hat{N}_{\phi, \psi} + \frac{(k - 2)}{2}}\right), \\ 
        &\| \pa_{t} u_{k + 2, p}\| _ {H^{s - 2(k + 4) + 4}} 
         \leq C_{\phi, \psi} \left(1  + (1 + t)^{\frac{k}{2} \hat{N}_{{\phi, \psi}} + \frac{(k-2)}{2}}\right) .
    \end{align*}
    Similar to the decomposition in \eqref{decomposition_f}, the  induction assumption and \eqref{gk+2est1} implies that
    \begin{align*}
        \| f ( u _ {a}) _ {k + 2, p} \| _ {H^{s  - 2(k + 4) + 4}} \leq C_{\phi, \psi} (1 + t)^{\frac{k + 2}{2} \hat{N}_{\phi, \psi} + \frac{k}{2}}.
    \end{align*}
    It follows that
    \begin{align*}
        \| u _ {k + 4, p} \| _ {H^{s - 2(k + 4) + 4}} \leq C_{\phi, \psi} (1 + t)^{\frac{k + 2}{2} \hat{N}_{\phi, \psi} + \frac{k}{2}}, \quad \forall \,p \geq 3.
    \end{align*}
    This is exactly \eqref{estd2} with $n = k + 4, j = 0$. The estimate of time derivatives follows similarly by induction argument. We complete the proof.
\end{proof}
Now we give a direct corollary of Propositions \ref{regd} and \ref{unpall}.
\begin{corollary} \label{existence}
    Under the assumptions in Proposition \ref{regd}, there exists a WKB solution $U_a$ of form \eqref{WKB_inital} with $U_{n,p} \in C([0, \infty); H^{s - 2n}(\mathbb{R}^2))$ satisfying all the properties in Propositions \ref{unplarge}, \ref{Un-1}, \ref{unpall}, \ref{initalUn}, \ref{regd} and Corollary \ref{poly}. In particular, for all $t \in (0, \infty)$ %% 这里最后写成 \begin{enumerate}
    \begin{align*}
         &\| U _ {n, 1}(t, \cdot) \| _ {H^{s -2n}} \leq C_{\phi, \psi} \left( 1  + ( 1 + t ) ^ {\frac{n}{2} \hat{N } _{\phi, \psi} + \frac{(n-2)}{2}} \right),\quad\forall \, n \in \left\{ 1,\, 2,\, \cdots,\, K+2\right\} \cap \mathbb{Z}_{\rm even}, \\
         &\|U _ {n, 1}(t, \cdot)\| _ {H^{s -2n + 1 } } 
         \leq C_{\phi, \psi} \left( 1 +(1 + t) ^{\frac{n - 1}{2} \hat{N}_{\phi, \psi}+\frac{(n-3)}{2}}\right), \quad\forall \, n \in \left\{1,\, 2,\, \cdots, K + 2\right\} \cap \mathbb{Z} _ {\rm odd}, \\
         &\| U_{n, p} (t, \cdot) \| _ {H^{s-2n + 4 }} 
         \leq C _ {\phi, \psi} \left( 1 + (1 + t)^{\frac{n - 2}{2} \hat{N}_{\phi, \psi} + \frac{(n - 4)}{2}}\right), \quad\forall \, p \geq 3,  n \in\left\{1,\, 2,\, \cdots,\, K+2\right\} \cap \mathbb{Z}_{\rm even}, \\
         &\| U _ {n, p} (t, \cdot) \| _ {H^{s - 2n + 5}} 
         \leq C_{\phi, \psi} \left( 1 + (1 + t )^{\frac{n - 3}{2} \hat{N}_{\phi, \psi} + \frac{(n - 5)}{2}}\right),  \quad\forall \, p \geq 3, n \in \left\{1,\, 2,\, \cdots,\,  K+2\right\} \cap \mathbb{Z}_{\rm odd}.
    \end{align*} %% 公式一般需要左边对齐
   % Additionally, the initial perturbation satisfies
    %\begin{align*}
   %     \|U(0, \cdot) - U_a(0, \cdot)\| _ {H^{s - 2K - 2}} \leq C_{\phi, \psi} {\color{red} \e^{K + 2} }.
   % \end{align*}
\end{corollary}

\subsection{Proof of Theorems \ref{f_remainder_estimate} and \ref{d_remainder_estimate}}

Now we prove Theorems \ref{f_remainder_estimate} and \ref{d_remainder_estimate}. Note that $U_a$ and $U_{n,p}$ take the same form in both focusing and defocusing case. The main difference is that $U_a$ can be globally constructed in the defocusing case. For the WKB solution $U_a$ in Propositions \ref{regf} and \ref{regd}, $\Phi_{n,p} = 0$ holds for all $n = -2,\, -1,\, 0,\, \cdots,\, K$ and $p \in \mathbb{Z}$. Then
\begin{align*}
    \pa_{t} \left (\sum_{n=0}^{K} \e^{n} U _ { n } \right) + \left( \sum_{n = K + 1}^{K + 2} \e^{n - 2} \sum_{p \in { {H}_{n}}} e ^ {\frac{ipt}{\e^{2}}} ( ip + U_{n, p}) \right) - \frac{1} {\e} A ( \pa _ {x} ) \left( \sum_{n = 0}^{K + 1} \e^{n} U_{n}\right) \\
    +\frac{1} {\e^{2}} A_{0} \left( \sum_{n = 0}^{K + 2} \e^{n} U _ {n} \right) = \sum_{n = 0}^{K} \e^{n} F(U_{a} )_{n}.
\end{align*}
    Here $U_a$ satisfies the equation %%这些超过的地方都不行
    \begin{align*}
        \pa_{t} U_{a} - \frac{1} {\e} A ( \pa _ {x}) U _ {a} + \frac{1}{\e^{2}} A_{0} U_{a} = F ( U _ {a} ) - \e^{K + 1} R _ {\e},
    \end{align*}
    with
    \begin{align*}
        R_{\e} := &\sum_{p \in H_{K + 2}} e ^ {\frac{ipt} {\e ^ {2}}} (\pa_{t} U _ {K  + 1, p} + \e \pa_{t} U_{K + 2, p} ) -A( \pa_{x} ) U_{K + 2}- \sum_{n = K + 2} ^ {3(K +2 )} \e^{n - K - 1} F(U_{a})_{n}.
    \end{align*}
    As shown in \eqref{inital per}, we have
    \begin{align*}
        U_{a} (0, \cdot) = U(0, \cdot) -  \e^{K + 2} r _ {\e}, \quad \text{where } r_{\e}:= - U_{K + 2}(0, \cdot).
    \end{align*}
    The estimates of $R_\e$ and $r_\e$ are given by the results in Propositions \ref{regf} and \ref{regd}.

\section{Stability of WKB solutions}\label{sec-stability}

In this section, we will show the stability of WKB solutions and prove Theorems \ref{f_ex}, \ref{d_ex_h} and \ref{d_ex_l}. We begin by introducing
\begin{align*} %%一般也还是用begin pmatrix
    \dot{U} = \begin{pmatrix}
      \dot{w} \\ 
      \dot{v} \\ 
      \dot{u}
    \end{pmatrix}: = \frac{U - U_{a}}{\e^{K+1}}.
\end{align*}
Then $\dot{U}$ satisfies the equation
\begin{align} \label{eq_dotu}
    \begin{cases}
      \pa_t \dot{U} - \frac{1}{\e} A(\pa_x) \dot{U} + \frac{1}{\e^2} A_0 \dot{U} = \frac{1}{\e^{K + 1}} ( F(U) - F(U_a) ) + R_\e, \\
      \dot{U}(0, \cdot) =   \e r_\e,
\end{cases}
\end{align}
where $R_\e$ and $r_\e$ satisfy the estimates in Theorem \ref{f_remainder_estimate} for the focusing case and Theorem  \ref{d_remainder_estimate} for the defocusing case. Observe that
\begin{align} \label{ob01}
\begin{split}
    \frac{1}{\e^{K+1}} ( f(u) - f(u_a) ) &=\la ( 3u_a^2 \dot{u} + 3\e^{K+1} u_a \dot{u}^2 + \e^{2(K+1)} \dot{u}^3 ).
\end{split}
\end{align}
We impose the following condition in the sequel
\begin{align*}
    \tilde{s} = s - 2K - 4 > 1.
\end{align*}

\subsection{Proof of Theorem \ref{f_ex}}

We start with the focusing case $\la < 0$. Recall that $T^*$ is the existence time of \eqref{ls_g_0}. Since \eqref{eq_dotu} is a symmetric hyperbolic system, there exists a unique solution $\dot{U} \in C([0, T^*_\e); H^{\tilde{s}})$ with $T^*_\e$ be the existence time of the solution. For all $t \in (0, \min\{T^*, T^*_\e\})$,  Duhamel's formula gives
\begin{align} \label{formUdot}
\begin{split}
    \dot{U}(t, \cdot) = 
    S(t) r_\e + \int_0 ^ t S(t - t') \frac{1}{\e^{K + 1}} ( F(U) - F(U_a) )(t', \cdot) \, \D{t'} + \int_0^t S(t - t') R_\e(t', \cdot) \, \D{t'},
\end{split}
\end{align}
with
\begin{align*}
    S(t) := \exp \left( ( \frac{1}{\e} A(\pa_x) - \frac{1}{\e^2} A_0 ) t \right).
\end{align*}
By the symmetry of $A(\pa_x)$ and $A_0$, we know that $S(t)$ is unitary from $H^s$ to $H^s$. Hence, by \eqref{ob01}, we have
\begin{align*}
    \|\dot{U}(t, \cdot)\|_{H^{\tilde{s}}} 
    \leq &\|r_\e\| _ {H^{\tilde{s}}} + \int_0^t \|R_\e(t', \cdot)\|_{H^{\tilde{s}}} \, \D{t'} \\
    &+ \int_0^t C ( \|u_a\|^2_{H^{\tilde{s}}} + \e^{K+1} \|u_a\| _ {H^{\tilde{s}}} \|\dot{u}\|_{H^{\tilde{s}}} + \e^{2(K+1)} \|\dot{u}\|^2_{H^{\tilde{s}}} ) \|\dot{u}\|_{H^{\tilde{s}}} \, \D{t'}.
\end{align*}
Using the estimates in Theorem \ref{f_remainder_estimate} and Proposition \ref{regf} gives
\begin{align*}
    \|\dot{U}(t, \cdot)\|_{H^{\tilde{s}}} 
    \leq C_{\phi, \psi} + t C_{\phi, \psi}(t) 
    + \int_0^t ( C_{\phi, \psi}(t') + C_{\phi, \psi}(t') \e^{K+1} \|\dot{u}\|_{H^{\tilde{s}}} + C \e^{2(K+1)} \|\dot{u}\|^2_{H^{\tilde{s}}} ) \|\dot{u}\|_{H^{\tilde{s}}} \, \D{t'}.
\end{align*}
We will show that $\liminf_{\e \to 0} T^*_\e \geq T^*$. Let $T < T^*$. It suffices to show there exists $\e_0 > 0$ such that for all $\e$ with $0 < \e < \e_0$, there holds $T^*_\e \geq T$. 
Define
\begin{align*}
    M(T) := ( C_{\phi, \psi} + T C_{\phi, \psi}(T) ) e^ { 2 T  C_{\phi, \psi}(T) + 2TC  },\quad\tilde{T}_1 := \sup \left\{ t < \min\{T, T^*_\e\} : \|\dot{U}\|_{L^\infty(0, t; H^{\tilde{s}})} \leq M(T) \right\} .
\end{align*}
Using Gronwall's inequality, we deduce
\begin{align*} 
\begin{split}
    \|\dot{U}\|_{L^\infty(0, t; H^{\tilde{s}})} 
    \leq &( C_{\phi, \psi} + t C_{\phi, \psi}(t) ) 
    e^{  t C_{\phi, \psi}(t) ( 1 + \e^{K + 1} \|\dot{u}\|_{L^\infty(0, t; H^{\tilde{s}})} ) + t C \e^{2(K+1)} \|\dot{u}\|^2_{L^\infty(0, t; H^{\tilde{s}})}} \\
    \leq &( C_{\phi, \psi} + T C_{\phi, \psi}(T) ) 
    e^{  T C_{\phi, \psi}(T) ( 1 + \e^{K+1} M(T) ) + T C \e^{2(K+1)} M(T)^2},\quad \forall \, t <\tilde{T}_1.
\end{split}
\end{align*}
Suppose $\e_0 > 0$ is small enough such that
\begin{align*}
    \e_0^{K+1} M(T) + \e_0^{2(K+1)} M(T)^2 = \frac{1}{2}.
\end{align*}
Then for all $\e$ with $0 < \e < \e_0$, there holds
\begin{align*}
    \|\dot{U}\|_{L^\infty(0, t; H^{\tilde{s}})} \leq ( C_{\phi, \psi} + T C_{\phi, \psi}(T) ) e^ {\frac{3}{2} T C_{\phi, \psi}(T) + \frac{1}{2} T C} \leq  M(T) -C(T),
\end{align*}
for some $C(T) >0$.  We thus can conclude that $$\tilde{T}_1  = \min\{T, T^*_\e\}.$$ 
Otherwise,  we may apply the continuity of $U$ on $[0, \tilde T]$ to imply $$\|\dot{U}\|_{L^\infty(0, \tilde T_{1}; H^{\tilde{s}})} \leq M(T)-C(T),$$ which contradicts to the definition of $\tilde{T}_1$. 

Therefore, for all $\e$ with $0 < \e < \e_0$, there holds that $T^*_\e > T$. Otherwise, we have $\tilde{T}_1 = T_{\e}^{*}$ and thus $\|\dot{U}\|_{L^\infty(0, T^*_\e ; H^{\tilde{s}})} \leq M(T),$ which contradicts to the definition of $T^{*}_{\e}$. 

 In addition, for all $t < \min\{T^*, T^*_\e\}$, there holds $\|\dot{U}(t, \cdot)\|_{H^{\tilde{s}}} \leq C_{\phi, \psi}(t)$. We thus complete the proof by  noting that $U - U_a = \e^{K+1} \dot{U}$. 
%\begin{remark} \label{cont_arg}
%    The proof of  $ \tilde{T}_1 = \min\{T, T^*_\e\}$ is called continuation argument. In the sequel, we will use it again. For brevity, we will omit the details.
%\end{remark}

\subsection{Proof of Theorem \ref{d_ex_h}}  \label {p2.4}%%这里也许要引用
For the defocusing case, we can construct a global-in-time WKB solution $U_a$. We will prove that $U_a$ is stable over long time. Consider the case with high regularity and let $K$ be a positive and even integer. We use \eqref{ob01} and \eqref{formUdot} to derive 
\begin{align} \label{estUdotd}
\begin{split}
    \|\dot{U} (t, \cdot)\| _ {H^{\tilde{s}}} 
    \le&  \|r_\e\| _ {H^{\tilde{s}}} + \int_0^t \|R_\e(t', \cdot)\|_{H^{\tilde{s}}} \, \D{t'}  
     + \int_0^t \la \tilde{M}  (\tilde{s} - 1) \big(( 3\| u_a \|_{H^{\tilde{s}}} \|u_a\| _ {L^\infty} \\
     &+3 \e^{K+1} \|u_a\|_{H^{\tilde{s}}} \| \dot{u} \| _ {H^{\tilde{s}}} + \e^{2(K + 1)} \| \dot{u} \| ^ 2 _ {H^{\tilde{s}}} ) \|\dot{u}\|_{H^{\tilde{s}}} \big) (t')\, \D{t'}, \quad \forall\, t \in (0, \infty),
\end{split}
\end{align}
where the prefactor $\tilde{M}(\tilde{s} - 1)$ is given by \eqref{estimate_def21}. 

By Propositions \ref{regd}, \ref{global_wp}, Corollary \ref{existence} and the classical estimates \eqref{classical}, we know that for all $   t \in (0, \infty)$, $u_a$ enjoys the two estimates below:
\begin{align} 
\begin{split}\label{est_uaH}
   & \|u_{a}(t, \cdot)\|_{H^{\tilde{s}}}= 
     \|(u_{0} + \e^{2} u_{2} + \cdots + \e^{K} u_{K} + \e^{K + 2} u_{K + 2})(t, \cdot)\|_{H^{\tilde{s}}} \\ 
     &\leq 2 C_{1}(\phi, \psi, \tilde{s})+ C_{\phi, \psi}\left(\e^{2}(1 + t)^{\hat{N}_{\phi, \psi}} + \e^{4}(1 + t)^{2 \hat{N}_{\phi, \psi} + 1} + \cdots  + \e^{K+2}(1 + t)^{\frac{K+2}{2} \hat{N}_{\phi, \psi} + \frac{K }{2}}\right),
\end{split} \\
\begin{split}\label{estuaL}
    &\| u _ {a}(t, \cdot)\| _ {L^{\infty}} = 
     \| (u_{0} + \e^{2} u_{2} + \cdots + \e^{K} u_{K} + \e^{K + 2} u_{K + 2}) (t, \cdot) \| _ {L^{\infty}} \\ 
    &\leq  2 C_{2}(\phi, \psi)(1 + t)^{-1} + C_{\phi, \psi} \left(\e^{2}(1 + t)^{\hat{N}_{\phi, \psi}} +  \cdots + \e^{K+2}(1 + t)^{\frac{K+2}{2} \hat{N}_{\phi, \psi} + \frac{K }{2}}\right).
\end{split}
\end{align}
Here $\hat{N}_{\phi, \psi}$ is defined in \eqref{definehatN}.
Let $0 < T_0 \leq 1$ be determined later. 

For brevity, we denote 
\begin{align*}
    &\e^{2}(1 + t)^{\hat{N}_{\phi, \psi}} + \e^{4}(1 + t)^{2 \hat{N}_{\phi, \psi} + 1} + \cdots  + \e^{K+2}(1 + t)^{\frac{K+2}{2} \hat{N}_{\phi, \psi} + \frac{K }{2}}=A, \\
    &\e^{2}(1 + t)^{\hat{N}_{\phi, \psi}} +  \cdots + \e^{K+2}(1 + t)^{\frac{K+2}{2} \hat{N}_{\phi, \psi} + \frac{K }{2}}=B.
\end{align*}
Substituting \eqref{est_uaH} and \eqref{estuaL} into \eqref{estUdotd} gives that
\begin{align} \label{U.est}
\begin{split}
        \|\dot{U}(t, \cdot)\|_{H^{\tilde{s}}} 
    &\leq C_{\phi, \psi} (1 + t)^{K} \left( 1 + (1 + t)^{\frac{K}{2} (\hat{N}_{\phi,\psi}-1)} \right) + \int_0^{\rm T} \left( C_{\phi, \psi} (A+B+AB) \right. \\
    & \left. +\tilde{N}_{\phi,\psi} (1 + t')^{-1} +  C_{\phi, \psi} \e^{K+1} (1+A)\|\dot{u}\|_{H^{\tilde{s}}} + C \e^{2(K+1)}  \|\dot{u}\|^2_{H^{\tilde{s}}} \right) \|\dot{u}\|_{H^{\tilde{s}}}(t') \, \D{t'}.
\end{split}
\end{align}
Here we define 
\begin{align} \label{define_tildeN}
    \tilde{N}_{\phi, \psi}:=12 \la \tilde{M}(s - 2 K - 5 ) C_{1}(\phi, \psi, s - 2 K - 4) C_{2}(\phi, \psi).
\end{align}
By imposing $N_{\phi,\psi} = \max\{\hat{N}_{\phi, \psi}, \tilde{N}_{\phi,\psi}\}$, we can substitute $\hat{N}$ and $\tilde{N}$ with $N$ in the formula above.
Moreover, we can naturally impose a time limit $T_\e$ as $T_\e = T_0 \e^{-\alpha}$ with  $\alpha = \frac{2}{N_{\phi,\psi}+1}$, $0 < T_0 \le 1$ to ensure $A= O(\e^\alpha)$ and $\, B=O(\e^\alpha)$. Hence, we can simplify the estimate \eqref{U.est}
\begin{align} \label{estUdoth}
\begin{split}
    \|\dot{U}(t, \cdot)\|_{H^{\tilde{s}}} 
    &\leq C_{\phi, \psi} (1 + t)^{K} \left( 1 + (1 + t)^{\frac{K}{2} (N_{\phi,\psi}-1)} \right) + \int_0^{\rm T} \left( N_{\phi,\psi} (1 + t')^{-1} \right. \\
    & \left. + C_{\phi, \psi} \e^{\alpha} + C_{\phi, \psi} \e^{K+1} \|\dot{u}\|_{H^{\tilde{s}}} + C \e^{2(K+1)} \|\dot{u}\|^2_{H^{\tilde{s}}} \right) \|\dot{u}\|_{H^{\tilde{s}}}(t') \, \D{t'}, \quad \forall\,t \in [0, T_\e]. 
\end{split}
\end{align}
Define 
$$\tilde{T}_3 := \sup \left\{ t \leq T_\e : \|\dot{U}\|_{L^\infty(0, t; H^{\tilde{s}})} \leq e^{3} C_{\phi, \psi} 2^{K} \e^{-K} \right\}.$$
By \eqref{definehatN} and \eqref{define_tildeN}, we directly have $N_{\phi,\psi} \geq 1$ and $\alpha \leq 1$. Then we can deduce  that for all $t \leq \tilde{T}_3$, the estimate \eqref{estUdoth} can be written as
\begin{align*}
    \|\dot{U}(t, \cdot) \|_{H^{\tilde{s}}} 
    \leq &2 C_{\phi, \psi}(1+t)^{\frac{K}{2} (N_{\phi, \psi} + 1)} + \int_{0}^{t}\left(N_{\phi, \psi} ( 1 + t' )^{-1} + \e^{\alpha} \hat{C}_{\phi, \psi} \right) \| \dot{u} (t', \cdot ) \|_{H^{\tilde{s}}} \D{t'}, %这里逗号不能忘记
\end{align*}
where
\begin{align*}
    \hat{C}_{\phi, \psi} := C_{\phi, \psi} + C_{\phi, \psi}^2 \e^{1 - \alpha} e^{3} 2^{K} +  C_{\phi, \psi}^2 \e^{2 - \alpha} e^{6} 2^{2K}.
\end{align*}
Using Gronwall's inequality gives
\begin{align*}
    &\frac{1}{2 C_{\phi, \psi}}\|\dot{U}(t, \cdot)\|_{H^{\tilde{s}}} \\ 
    &\le (1 + t)^{\frac{K}{2} ( N_{\phi, \psi}+1)}+\int_{0}^{t} ( 1 + t' ) ^ {\frac{K}{2} (N_{\phi, \psi} + 1)} ( \frac{N_{\phi, \psi}}{1 + t'} + \e^{\alpha} \hat{C}_{\phi, \psi}) e^{\int_{t'}^{t}\left(\frac{N_{\phi, \psi}}{1 + t^{''}} + \e^{\alpha} \hat{C}_{\phi, \psi}\right) \D{t''}} \D{t'} \\ 
    &\le (1 + t)^{\frac{K}{2}(N_{\phi, \psi} + 1)}  +\int_{0}^{t}( 1 + t' )^{\frac{K}{2}(N_{\phi, \psi} + 1)}\left(\frac{N_{\phi, \psi}}{1 + t'} + \e^{\alpha} \hat{C}_{\phi, \psi}\right)\left(\frac{1 + t}{1 + t'}\right)^{N_{\phi, \psi}} e^{T_{0} \hat{C}_{\phi, \psi}} \D{t'} \\ 
    &\le  (1+t)^{\frac{K}{2}(N_{\phi, \psi}+1)}+(1+T_{0} \hat{C}_{\phi, \psi}) e^{T_{0} \hat{C}_{\phi, \psi}}(1+t)^{\frac{K}{2}(N_{\phi, \psi}+1)}, \quad \forall\, t \leq \tilde{T}_3.
\end{align*}
For all $t \leq T_\e$, there holds 
\begin{align*}
    C_{\phi, \psi}(1 + t)^{\frac{K}{2} N_{\phi,\psi} + \frac{K}{2}}  \leq 2 C_{\phi, \psi} 2^{K} \e^{-K}.
\end{align*}
Choose $T_0$ such that
\begin{align} \label{T_0limitb}
    T_0 \hat{C}_{\phi, \psi} = T_0 ( C_{\phi, \psi} + C _{\phi, \psi}^2 \e^{1 - \alpha} e^{3} 2^{K} + C_ {\phi, \psi}^2 \e^{2 - \alpha} e^{6}  2^{2K} ) \leq \frac{1}{2}.
\end{align}

Hence,
\begin{align*}
    \|\dot{U}(t, \cdot)\|_{H^{\bar{s}}} \leq 5 C_{\phi, \psi}(1 + t)^{\frac{K}{2}(N_{\phi, \psi} + 1)}  e^{\frac{1}{2}} \leq C_{\phi, \psi} e^{\frac{5}{2}} \frac{2^{K}}{\e^{K}}, \quad \forall\, t \leq \tilde{T}_3.
\end{align*}
Using continuation argument gives that $\tilde{T}_3 = T_\e$.

Therefore, for all $T_0$ chosen in \eqref{T_0limitb} and  $t \leq \frac{T_0}{\e^\alpha}$, there holds
\begin{align*}
    \|\dot{U}(t, \cdot)\|_{H^{\tilde{s}}} \leq 5 e^{\frac{1}{2}} C_{\phi, \psi} (1 + t)^{\frac{K}{2}(N_{\phi, \psi} + 1)}.
\end{align*}

We complete the proof.

\subsection{Proof of Theorem \ref{d_ex_l}}

The proof of Theorem \ref{d_ex_l} is similar to the arguments in Section \ref{p2.4}. We consider a WKB solution taking the form in \eqref{WKB_inital} with $K = 0$. Hence, by \eqref{U0form}, \eqref{U1form}, \eqref{U2form}, we impose $g_1 = g_2 = 0$ to derive
\begin{align*}
    U_a = U_0 + \e U_1 + \e^2 U_2. 
\end{align*}
Here, $U_0$, $U_1$ and $U_2$ are respectively given by \eqref{U0form}, \eqref{U1form} and \eqref{U2form}, where  $g_0$ is the solution to the defocusing cubic Schr\"{o}dinger equation \eqref{ls_g_0} and satisfies the estimates in Proposition \ref{global_wp}.

The estimate of $\dot{U}$ follows from \eqref{estUdotd} with $K=0$ in $\tilde{s}$, and we can apply  Theorem \ref{d_remainder_estimate} to get the estimates of $R_\e$ and $r_\e$.

Moreover, with \eqref{U0form}, \eqref{U1form} and \eqref{U2form}, it is clear that
\begin{align*}
    \|u_a(t, \cdot)\|_{H^{s - 4}} = \|(u_0 + \e^2 u_2)(t, \cdot)\|_{H^{s - 4}} \leq 2C_1(\phi, \psi, s - 4) + C_{\phi, \psi} \e^2, \\
    \|u_a(t, \cdot)\|_{L^\infty} = \|(u_0 + \e^2 u_2)(t, \cdot)\|_{L^\infty} \leq 2C_2(\phi, \psi) (1 + t)^{-1} + C_{\phi, \psi} \e^2.
\end{align*}
Then
\begin{align*}
    \|\dot{U}(t, \cdot)\|_{H^{s-4}} \leq &C_{\phi, \psi}(1 + t) + \int_{0}^{t}\left(\tilde{N}_{\phi, \psi}( 1 + t' )^{-1} + C_{\phi, \psi} \e^{2} + C_{\phi, \psi}\|\dot{u}\|_{H^{s-4}}\right. \\ 
    &\left. + C \e^{2} \| \dot{u} \| _ {H^{s - 4}} ^{2}\right)\|\dot{u}\|_{H^{s - 4}}(t') \D{t'}.
\end{align*}

The remaining part follows similarly to the proof of Theorem \ref{d_ex_h}. 

\subsection{Error estimates in the non-relativistic regime} \label{error}
 %% 这里几句话 单独一章 太少了, 合并上面去应该要. 好的，格式改完我并上去

Now we show Theorems \ref{defo_1}, \ref{defo_2}, \ref{defo_3} and \ref{fo_1}. In fact, they are corollaries of Theorems \ref{f_ex}, \ref{d_ex_h} and \ref{d_ex_l} . 

Theorem \ref{defo_3} is a corollary of Theorem \ref{d_ex_h} . Theorem \ref{defo_2} is a corollary of Theorem \ref{d_ex_l}. Theorem \ref{fo_1} is a corollary of Theorem \ref{f_ex}. The proofs are straightforward, so we will omit them.

Theorem \ref{defo_1} is a special case of Theorem \ref{defo_3} with $K = 2$ and $s > 9 $. By \eqref{defo_3_estimate1} in Theorem \ref{defo_3}, there holds
\begin{align*}
    \|(u - u_0 - \e^2 u_2 - \e^4 u_4)(t, \cdot)\|_{H^{s - 8}(\mathbb{R}^2)} \leq C_{\phi, \psi} (1 + t)^{N_{\phi, \psi} + 1}  \e^3, \, t \leq T_0 \e^{-\alpha}, 
\end{align*}
where $u_0, \, u_2,\, u_4$ satisfy \eqref{form_Ua1} and \eqref{defo_3_estimate2}. Thus we have
\begin{align*}
    \|(u - u_0)(t, \cdot)\|_{H^{s - 8}} 
    \le& \e^2 \|u_2(t, \cdot)\|_{H^{s - 8}} + \e^4 \|u_4(t, \cdot)\|_{H^{s - 8}} + C_{\phi, \psi} (1 + t)^{N_{\phi, \psi} + 1} \e^3 \\
   \leq & C_{\phi, \psi} (1 + t)^{N_{\phi, \psi}} \e^2 \left( 1 + \e^2 + (1 + t) \e \right) \\
    \le& C_{\phi, \psi}(1 + t)^{N_{\phi, \psi}}  \e^2, \quad \forall \, t \leq T_0 \e^{-\a}.
\end{align*}
This is exactly Theorem \ref{defo_1}.

\end{document}